\documentclass{ws-igtr}

\usepackage{epsfig,epsf,fancybox}
\usepackage{amsmath}
\usepackage{mathrsfs,bbm}
\usepackage{amssymb}
\usepackage{graphicx}
\usepackage{color}
\usepackage{multirow}
\usepackage{paralist}
\usepackage{verbatim}
\usepackage{galois}
\usepackage{algorithm}
\usepackage{algorithmic}
\usepackage{boxedminipage}
\usepackage{booktabs}
\usepackage{accents}
\usepackage{stmaryrd}
\usepackage{subfig}
\usepackage{appendix}
\usepackage{graphicx}
\usepackage{epstopdf}
\usepackage{appendix}
\usepackage{cases}

\newcommand{\lr}{\mathcal{L}}
\newcommand{\e}{\mathbb{E}}
\newcommand{\br}{\mathbb{R}}
\newcommand{\pr}{\mathcal{P}}
\newcommand{\dd}{\partial}

\newcommand{\brn}{{\mathbb{R}^n}}

\newcommand{\brd}{{\mathbb{R}^d}}

\newcommand{\hv}{\hat{V}}

\newcommand{\tv}{\tilde{V}}
\newcommand{\tx}{\tilde{X}}

\newcommand{\hx}{\hat{X}}
\newcommand{\hp}{\hat{P}}
\newcommand{\hq}{\hat{Q}}
\newcommand{\hr}{\mathcal{H}}
\newcommand{\ur}{\mathcal{U}}

\newcommand{\bx}{\bar{X}}

\newcommand{\argmin}{\mathop{\rm argmin}}

\allowdisplaybreaks[2]

\begin{document}



\title{Maximum Principle for Mean Field Type Control Problems with General Volatility Functions}

\author{Alain Bensoussan}

\address{International Center for Decision and Risk Analysis,
	Naveen Jindal School of Management, University of Texas at Dallas, USA\\ 
\email{axb046100@utdallas.edu} }

\author{Ziyu Huang}

\address{School of Mathematical Sciences, Fudan University, China \\
\email{zyhuang19@fudan.edu.cn} }

\author{Sheung Chi Phillip Yam}

\address{Department of Statistics, The Chinese University of Hong Kong, Hong Kong, China\\ 
	\email{scpyam@sta.cuhk.edu.hk} }

\maketitle


\begin{abstract}
In this paper, we study the maximum principle of mean field type control problems when the volatility function depends on the state and its measure and also the control, by using our recently developed method in [\cite{AB6}, \cite{AB5}, \cite{AB}]. Our method is to embed the mean field type control problem into a Hilbert space to bypass the evolution in the Wasserstein space. We here give a necessary condition and a sufficient condition for these control problems in Hilbert spaces, and we also derive a system of forward-backward stochastic differential equations. 
\end{abstract}


\ccode{Subject Classification:  35R15, 49L25, 49N70, 91A13,
93E20, 60H30; 60H10.}

\section{Introduction}

Mean field type control problems and mean field games have received  much attention in recent years. The literature in this area is huge. We refer to some recent relevant literature such as Cardaliaguet-Delarue-Lasry-Lions [\cite{CDLL}], Carmona-Delarue [\cite{book_mfg}], Bensoussan-Frehse-Yam [\cite{AB_book}], Cosso-Pham [\cite{CA}], Pham-Wei [\cite{PH}], Buckdahn-Li-Peng-Rainer [\cite{BR}], Djete-Possamai-Tan [\cite{DMF}], Gomes-Pimentel-Voskanyan [\cite{GDA}], Huang-Malham\'e-Roland-Caines [\cite{HM}], Porretta [\cite{PA}], Carmona-Delarue [\cite{CR}], Gangbo-M\'esz\'aros-Mou-Zhang [\cite{GW}], Chassagneux-Crisan-Delarue [\cite{CJF}], Cardaliaguet-Cirant-Porretta [\cite{CP1}], Bensoussan-Yam [\cite{AB}], Bensoussan-Graber-Yam [\cite{AB7}], and Bensoussan-Tai-Yam [\cite{AB5}].

We here study the mean field type control problem with a general volatility function, following our previous work [\cite{AB6}, \cite{AB5}, \cite{AB}]. Our approach is to embed the mean field control type problem into one in Hilbert spaces, which looks like the lifting method proposed by P. L. Lions, but actually they are fundamentally different. This lifting method was used in Bensoussan-Yam [\cite{AB}] to study the control problem in the Wasserstein space. In [\cite{AB6}], we used a different Hilbert space and this allows to recover the mean field type control problem as a particular case. This kind of Hilbert space is already used in [\cite{AB5}], for the case when the  volatility is a matrix, and the cost functional is separable. 

This article is organized as follows. In Section~\ref{sec:2}, we introduce the Wasserstein space, the derivatives of functionals, the Hilbert space and the Wiener process. Section~\ref{sec:3} gives the formulation of our problem. In Section~\ref{sec:Hilbert}, we study control problems in Hilbert spaces and give a necessary condition in view of the maximum principle. In Section~\ref{sec:FB_Hilbert}, we give a sufficient condition by showing the cost functional is convex. We also derive a system of forward-backward stochastic differential equations (FBSDEs) defined in Hilbert spaces, the solution of which gives the optimal control. In Section~\ref{sec:MFC}, we go back to mean field type control problems and interpret our results for control problem in Hilbert spaces as those for mean field type control problems. 

\section{Problem Setting and Notations}\label{sec:2}

\subsection{Wasserstein space}

We consider the space $\mathcal{P}_{2}(\brn)$ of probability measures on $\brn$, each with a second moment, equipped with the 2-Wasserstein metric defined by
\begin{align*}
	W_2(m,m'):=\inf_{\pi\in\Gamma(m,m')}\sqrt{\int_{\brn\times\brn}|x-x'|^2\pi(dx,dx')},
\end{align*}
where $\Gamma(m,m')$ denotes the set of joint probability measures with respective marginals $m$ and $m'$. The infimum is attained, so we can find $\hat{X}_{m},\hat{X}_{m'}$ in $L^{2}(\Omega,\mathcal{A},\mathbb{P};\brn)$, where $(\Omega,\mathcal{A},\mathbb{P})$ is an atomless probability space, whose probability laws $\mathcal{L}X_{m}=m$, $\mathcal{L}X_{m'}=m'$, such that 
\begin{align*}
	W_{2}^{2}(m,m')=\e|\hat{X}_{m}-\hat{X}_{m'}|^{2}. 
\end{align*}
A family $m_{k}$ converges to $m$ in $\mathcal{P}_{2}(\brn)$ if and only if it converges in the sense of the weak convergence and $W_2(m_k,\delta_0)\to W_2(m,\delta_0)$ as $k\to\infty$, see Ambrosio-Gigli-Savar\'e [\cite{AL}] and Villani [\cite{VC}] for details.

\subsection{Derivatives of functionals}

We consider a functional $F:\mathcal{P}_{2}(\brn)\to\br$. For the concept of derivative, we use the concept of the linear functional derivative; see Carmona-Delarue [\cite{book_mfg}]. The linear functional derivative of $F(m)$ at $m$ is a function $\mathcal{P}_{2}(\brn)\times \brn\ni(m,x)\mapsto\dfrac{dF}{d\nu}(m)(x)$ being continuous under the product topology, satisfying 
\begin{align*}
	\int_{\brn}\left|\dfrac{dF}{d\nu}(m)(x)\right|^{2}dm(x)\leq c(m),
\end{align*}
for some positive constant $c(m)$ depending locally on $m$, and for any $m'\in\mathcal{P}_{2}(\brn)$,
\begin{align*}
	\lim_{\epsilon\to0}\dfrac{F(m+\epsilon(m'-m))-F(m)}{\epsilon}=\int_\brn\dfrac{dF}{d\nu}(m)(x)(dm'(x)-dm(x)).
\end{align*}
This definition implies that
\begin{align*}
	\dfrac{d}{d\theta}F(m+\theta(m'-m))=\int_{\brn}\dfrac{dF}{d\nu}(m+\theta(m'-m))(x)(dm'(x)-dm(x))
\end{align*}
and 
\begin{align*}
	F(m')-F(m)=\int_{0}^{1}\int_{\brn}\dfrac{dF}{d\nu}(m+\theta(m'-m))(x)(dm'(x)-dm(x))d\theta. 
\end{align*}
Here, we write $\dfrac{dF}{d\nu}(m)(x)$ instead of $\dfrac{dF}{dm}(m)(x)$ to make the difference between the notation $\nu$ and the argument $m.$ Also we prefer $\dfrac{dF}{d\nu}(m)(x)$ to $\dfrac{\delta F}{\delta m}(m)(x)$ used in Carmona-Delarue [\cite{book_mfg}], because there is no risk of confusion and it works pretty much like an ordinary G\^ateaux derivative in the space of measures with $L^2$-integrable densities. 

\subsection{Hilbert spaces}

In this paper, the control is $\brd$-valued, and the state is $\brn$-valued. For $m\in\pr_2(\brn)$, consider $L^2_m (\brn;\brn)$ the Hilbert space with respect to the following inner product
\begin{align*}
	(X,X'):=\int_{\brn} X_x^* X'_xdm(x),\quad X,X'\in L^2_m (\brn;\brn).
\end{align*}
We also work in Hilbert spaces $\hr_m:=L^2(\Omega,\mathcal{A},\mathbb{P};L_m^2(\brn;\brn))$ and $\ur_m:=L^2(\Omega,\mathcal{A},\mathbb{P};L_m^2(\brn;\brd))$. An element of $\hr_m$ is denoted by $X_x(\omega),\ (x,\omega)\in\brn\times\Omega$; for any $x\in\brn$, $X_x$ is a $\brn$-random vector in $L^2(\Omega,\mathcal{A},\mathbb{P};\brn)$. Similarly, an element of $\ur_m$ is denoted by $V_x(\omega),\ (x,\omega)\in\brn\times\Omega$; for any $x\in\brn$, $V_x$ is a $\brd$-random vector in $L^2(\Omega,\mathcal{A},\mathbb{P};\brd)$. The inner products over $\hr_m$ and $\ur_m$ are respectively defined as
\begin{align*}
	&(X,X')_{\hr_m}:=\e\left[\int_\brn X_x^* X'_x dm(x)\right],\quad X,X'\in\hr_m,\\
	&(V,V')_{\ur_m}:=\e\left[\int_\brn V_x^* V'_x dm(x)\right],\quad V,V'\in\ur_m.
\end{align*}
For $X\in\hr_m$, we denote by $X\#(m\otimes\mathbb{P})$ the push-forward of $m\otimes\mathbb{P}$ by the map $(x,\omega)\mapsto X_x(\omega)$. For test function $\varphi$, we have
\begin{align*}
	\int_\brn \varphi(y)d(X\#(m\otimes\mathbb{P}))(y)=\e \left[\int_\brn \varphi(X_x)dm(x)\right].
\end{align*}
We have $X\#(m\otimes\mathbb{P})\in\pr_2(\brn)$, and
\begin{equation}\label{m2}
	\begin{aligned}
		&W_2(X\#(m\otimes\mathbb{P}),\delta_0)=\|X\|_{\hr_m},\\
		&W_2(X\#(m\otimes\mathbb{P}),X'\#(m\otimes\mathbb{P}))\le \|X-X'\|_{\hr_m},\quad X,X'\in \hr_m.
	\end{aligned}	
\end{equation}
We refer to [\cite{AB6}, \cite{AB5}] for details. For notational convenience, we use the notation $X\otimes m$ instead of $X\#(m\otimes\mathbb{P})$.

\subsection{Wiener process and filtrations}

Let $w(t),\ t\geq 0$ be an $n$-dimensional standard Wiener process on $(\Omega,\mathcal{A},\mathbb{P})$. For any $0\le t\le s\le T$, we denote by $\mathcal{W}_t^s:=\sigma(w(\tau)-w(t);\tau\in[t,s])$ the filtration generated by the Wiener process, and denote by $\mathcal{W}_t:=\mathcal{W}_t^T$. For any $X\in\hr_m$ independent of $\mathcal{W}_t$, we denote by $\mathcal{W}^s_{tX}:=\sigma(X)\bigvee \mathcal{W}_t^s$ the filtration generated by $X$ and the Wiener process, and denote by $\mathcal{W}_{tX}$ the filtration generated by the $\sigma$-algebras $\mathcal{W}_{tX}^s$ for all $s\in[t,T]$. We denote by $L^2_{\mathcal{W}_{tX}(t,T;\hr_m)}$ the subspace of $L^2(t,T;\hr_m)$ of all processes adapted to the filtration $\mathcal{W}_{tX}$.

\section{Formulation of the Control Problem}\label{sec:3}

\subsection{Problem formulation}

We first introduce the dynamical system. We consider maps
\begin{align*}
	&\text{Drift:}\quad\qquad  f:\brn\times\mathcal{P}_{2}(\brn)\times \brd\times[0,T]\rightarrow \brn,\\
	&\text{Volatility:}\quad \sigma: \brn\times\mathcal{P}_{2}(\brn)\times \brd\times [0,T]\rightarrow \br^{n\times n}.
\end{align*}
For $(t,m)\in[0,T]\times\pr_2(\brn)$ and initial $X\in \hr_m$ independent of $\mathcal{W}_t$, the space of controls is set as the Hilbert space $L^2_{\mathcal{W}_{tX}}(t,T;\ur_m)$, and a control is denoted by $v_x(s)$. Then, the state process $X^v(\cdot)$ associated with $v(\cdot)$ is defined as
\begin{equation}\label{eq:3-21}
	\begin{aligned}
		X_x^v(s)=&X_x+\int_t^s f(X_x^v(r),X^v(r)\otimes m,v_x(r),r)dr\\
		&+\sum_{j=1}^{N} \int_t^s \sigma^j(X_x^v(r),X^v(r)\otimes m,v_x(r),r)dw_j(r),\quad s\in[t,T].
	\end{aligned}	
\end{equation}
The state process $X^v(\cdot)$ belongs to $L^2_{\mathcal{W}_{tX}}(t,T;\hr_m)$ under suitable assumptions on coefficients $g$ and $\sigma$ to be stated in Section~\ref{sec:Hilbert}. For $s\in[t,T]$, $X^v(s)\otimes m$ is the probability law of $X^v_\eta(s)$ when $\eta$ is equipped with a probability $m\otimes\mathbb{P}$. We then introduce the loss function. We consider the respective running and terminal loss coefficient functions
\begin{align*}
	&g:\brn\times\mathcal{P}_{2}(\brn)\times \brd\times [0,T]\rightarrow \br,\\
	&g_T:\brn\times\mathcal{P}_{2}(\brn)\rightarrow \br.
\end{align*}
We then define the loss functional as
\begin{equation}\label{eq:3-22}
	\begin{aligned}
		J_{tX}(v(\cdot)):=&\int_{t}^{T}\e\left[\int_{\brn}g(X^v_x(s),X^v(s)\otimes m,v_x(s),s)dm(x)\right]ds\\
		&+\e\left[\int_{\brn}g_T(X^v_x(T),X^v(T)\otimes m)dm(x)\right],\quad v(\cdot) \in L^2_{\mathcal{W}_{tX}}(t,T;\ur_m).
	\end{aligned}
\end{equation}
Then, our control problem is defined as
\begin{align*}
	\inf_{v \in L^2_{\mathcal{W}_{tX}}(t,T;\ur_m)}J_{tX}(v).
\end{align*}

\subsection{Deterministic control problem in Hilbert space}

In our previous work [\cite{AB6}, \cite{AB9}], we study the deterministic control problem. For $(t,m)\in [0,T]\times\pr_2(\br^n)$ and $X\in L^2_m(\brn;\brn)$, the McKean-Vlasov control problem is defined as:
\begin{equation}\label{intr_5}
	\begin{aligned}
		&\inf_{v\in L^2([t,T];L_m^2(\brn;\brd))} \ J_{tX}(v),\\
		&\text{where}\quad  X_{x}^{v}(s)={X}_x+\int_t^s f(X^{v}_{x}(r),X^{v}_{\cdot}(r)\# m,v_{x}(r),r)dr,\quad s\in[t,T],
	\end{aligned}
\end{equation}
where the loss functional is defined as
\begin{align*}
	J_{tX}(v(\cdot)):=&\int_t^T \int_{\br^n} g(X_{x}^{v}(s),X^{v}_{\cdot}(s)\# m,v_{x}(s),s)dm(x)ds\\
	&+\int_{\br^n} g_T(X_{x}^{v}(T),X^{v}_{\cdot}(T)\# m)dm(x),\quad v(\cdot)\in L^2([t,T];L_m^2(\brn;\brd)).
\end{align*}
To obtain the optimality condition, we introduce the Lagrangian $L:\brn\times\pr_2(\brn)\times\brd\times[0,T]\times\brn\to\br$ as
\begin{align*}
	L(x,m,v,s;p):=p^* f(x,m,v,s)+g(x,m,v,s).
\end{align*}
Let $\hat{v}_x(s)$ be an optimal control for \eqref{intr_5} and $\hx_x(s)$ be the corresponding state process. The adjoint process is defined by, for $s\in[t,T]$,
\begin{align*}
	\hp_x(s)=&D_xg_T\left(\hx_x(T),\hx_\cdot(T)\#m\right)+\int_\brn D_\xi\frac{d g_T}{d\nu}\left(\hx_y(T),\hx_\cdot(T)\#m\right)\left(\hx_x(T)\right) dm(y)\\
	&+\int_s^T \bigg[D_x {L}\left(\hx_x(r),\hx_\cdot(r)\#m,\hat{v}_x(r),r;\hp_x(r)\right)\\
	&\quad\qquad+\int_\brn D_\xi \frac{d L}{d\nu}\left(\hx_y(r),\hx_\cdot(r)\#m,\hat{v}_y(r),r;\hp_y(r)\right)\left(\hx_x(r)\right) dm(y)\bigg]dr.
\end{align*}
Then, the optimal control satisfies the optimality condition
\begin{align*}
	D_v L\left(\hx_x(s),\hx_\cdot(s)\#m,\hat{v}_x(s),s;\hp_x(s)\right)=0,\quad \text{a.e.}\ s\in[t,T],\  \text{a.s.}\ dm(x).
\end{align*}

\section{Necessary Condition for Control Problem \eqref{intr_3} on Hilbert Spaces}\label{sec:Hilbert}

The formulation of problem \eqref{eq:3-21}-\eqref{eq:3-22} inspires us to study control problems on Hilbert spaces $\hr_m$ and $\ur_m$. For any initial time $t\in[0,T]$ and $X\in\hr_m$ (being independent of $\mathcal{W}_t$), in this section, we consider the following control problem
\begin{equation}\label{intr_3}
	\begin{aligned}
		&\inf_{V\in L^2_{\mathcal{W}_{tX}}(t,T;\ur_m)} \ J_{tX}(V):=\int_t^T G\left(X^{V}(s),V(s),s\right) ds+G_T\left(X^{V}(T)\right),\\
		&\text{s.t.}\   X^{V}(s)=X+\int_t^s F\left(X^{V}(r),V(r),r\right)dr+\sum_{j=1}^{n} A^j\left(X^{V}(r),V(r),r\right)dw_j(r),
	\end{aligned}
\end{equation}
where 
\begin{align*}
	&F:\hr_m\times \ur_m\times[0,T]\to \hr_m,\quad A:\hr_m\times \ur_m\times[0,T]\to \hr_m^n,\\
	&G:\hr_m\times \ur_m\times[0,T]\to \br,\quad G_T:\hr_m\to \br.
\end{align*}
We shall make a connection between mean field type control problem \eqref{eq:3-21}-\eqref{eq:3-22} and the abstract control problem in (certain suitably chosen) Hilbert spaces \eqref{intr_3} in Section~\ref{sec:MFC}. We state our assumptions for necessary condition for \eqref{intr_3}. We shall make a connection between assumptions for mean field type control problem \eqref{eq:3-21}-\eqref{eq:3-22} and the following assumptions in Section~\ref{sec:MFC}. For notational convenience, we use the same constant $L>0$ for all the conditions below. \\
\textbf{(A1)} The maps $F$ and $A^j$ (for $1\le j\le n$) satisfy for $(X,V,s)\in \hr_m\times \ur_m\times[0,T]$,
\begin{align*}
	&\|F(X,V,s)\|_{\hr_m}\le L(1+\|X\|_{\hr_m}+\|V\|_{\ur_m}),\\
	&\|A^j(X,V,s)\|_{\hr_m}\le L(1+\|X\|_{\hr_m}+\|V\|_{\ur_m}).
\end{align*}    
For any $(X,V,s)$, the G\^ateaux derivatives of $F$ and $A^j$ along any directions $\tx\in \hr_m$ and $\tv\in \ur_m$ at $(X,V,s)$ exist and are continuous in $(X,V)$, and satisfy
\begin{align*}
	&\|D_X F(X,V,s)(\tx)\|_{\hr_m}\le L\|\tx\|_{\hr_m},\quad \|D_V F(X,V,s)(\tv)\|_{\hr_m}\le L\|\tv\|_{\ur_m},\\
	&\|D_X A^j(X,V,s)(\tx)\|_{\hr_m}\le L\|\tx\|_{\hr_m},\quad \|D_V A^j(X,V,s)(\tv)\|_{\hr_m}\le L\|\tv\|_{\ur_m}.
\end{align*}
\textbf{(A2)} The running cost functional $G$ satisfies
\begin{align*}
	|G(X,V,s)|\le L(1+\|X\|^2_{\hr_m}+\|V\|^2_{\ur_m}),\quad (X,V,s)\in \hr_m\times \ur_m\times[0,T].
\end{align*}
For any $s\in [0,T]$, the functional $\hr_m\times \ur_m\ni(X,V)\mapsto G(X,V,s)\in\br$ is continuously differentiable, with its derivatives
\begin{align*}
	\|D_X G (X,V,s)\|_{\hr_m}+\|D_V G (X,V,s)\|_{\ur_m} \le L(1+\|X\|_{\hr_m}+\|V\|_{\ur_m}).
\end{align*}
The terminal cost functional $G_T$ satisfies
\begin{align*}
	|G_T(X)|\le L(1+\|X\|^2_{\hr_m}),\quad X\in \hr_m,
\end{align*}   
and is continuously differentiable, with its derivative 
\begin{align*}
	\|D_XG_T(X)\|_{\hr_m} \le L(1+\|X\|_{\hr_m}).
\end{align*}
For any control $V\in L^2_{\mathcal{W}_{tX}}(t,T;\ur_m)$, we study the regularity in $V$ of the corresponding controlled state process $X^V(s)$. The following result is proven in \ref{pf:lem:05}.

\begin{lemma}\label{lem:05}
	Under Assumption (A1) and $V\in L^2_{\mathcal{W}_{tX}}(t,T;\ur_m)$, the corresponding controlled state process $X^V$ belongs to $L^2_{\mathcal{W}_{tX}}(t,T;\hr_m)$ and satisfies
	\begin{align}\label{lem02_1}
		\sup_{t\le s\le T}\left\|X^V(s)\right\|_{\hr_m}\le C(L,T)(1+\|X\|_{\hr_m}+\|V\|_{L^2(t,T;\ur_m)}),
	\end{align}
	where $C(L,T)$ is a constant depending only on $(L,T)$. For $\tv\in L^2_{\mathcal{W}_{tX}}(t,T;\ur_m)$, we set $V^\epsilon:=V+\epsilon \tv$ for $\epsilon\in[0,1]$ and denote by $X^\epsilon$ the state process corresponding to the control $V^\epsilon$. Then, we have
	\begin{align}\label{lem05_1}
		\lim_{\epsilon\to0}\sup_{t\le s\le T}\left\|\frac{1}{\epsilon}[X^\epsilon(s)-X^V(s)]-\mathcal{D}_{\tv}X^V(s)\right\|_{\hr_m}=0,
	\end{align}
	where $\tx^{V,\tv}\in L^2_{\mathcal{W}_{tX}}(t,T;\hr_m)$ the unique solution of the following equation
	\begin{equation}\label{tx}
		\begin{aligned}
			\mathcal{D}_{\tv}X^V(s)=&\int_t^s \Big[D_X F \left(X^V(r),V(r),r\right)\left(\mathcal{D}_{\tv}X^V(r)\right)\\
			&\quad\qquad +D_V F\left(X^V(r),V(r),r\right)\left(\tv(r)\right)\Big]dr\\
			&+\sum_{j=1}^n\int_t^s \Big[D_X A^j \left(X^V(r),V(r),r\right)\left(\mathcal{D}_{\tv}X^V(r)\right)\\
			&\quad\qquad\qquad +D_V A^j\left(X^V(r),V(r),r\right)\left(\tv(r)\right)\Big]dw_j(r),\quad s\in[t,T],
		\end{aligned}
	\end{equation}
	such that
	\begin{align}\label{lem04_1}
		\sup_{t\le s\le T}\left\|\mathcal{D}_{\tv}X^V(s)\right\|_{\hr_m}\le C(L,T)\left\|\tv\right\|_{L^2(t,T;\ur_m)}.
	\end{align}
	That is, $\mathcal{D}_{\tv}X^V(s)$ is actually the directional derivative ${D}_{\tv}X^V(s)$ of $X^V(s)$ along the direction $\tv\in L^2(t,T;\ur_m)$.
\end{lemma}

Lemma \ref{lem:05} shows that the mapping
\begin{align*}
	L^2(t,T;\ur_m)\ni V\mapsto X^{V}(s)\in \hr_m
\end{align*}
is G\^ateaux differentiable. The directional derivative $D_{\tv}X^V$ evaluated at $V$ can be expressed as a Fr\'echet derivative $D_V X^V$ acting on $\tv$, i.e. $D_V X^V(\tv)$, see [\cite{AB5}, \cite{AB9}]. Here, the subscript $\tv$ in the directional derivative $D_{\tv}X^V$ means the direction while that $V$ in $D_V X^V(\tv)$ means the evaluating point in the Fr\'echet derivative, we wish this may not cause too much ambiguity. We introduce the Lagrangian $\lr:\hr_m\times \ur_m\times[0,T]\times\hr_m\times\hr_m^n\to\br$ such that for $(X,V,s,P,Q)\in \hr_m\times \ur_m\times[0,T]\times \hr_m\times\hr_m^n$,
\begin{equation}\label{H}
	\begin{aligned}
		\lr(X,V,s;P,Q):=\left(P,F(X,V,s)\right)_{\hr_m}+\sum_{j=1}^n \left(Q^j,A^j(X,V,s)\right)_{\hr_m} +G(X,V,s).
	\end{aligned}
\end{equation}
Then, we have
\begin{align*}
	D_P\lr (X,V,s;P,Q)=F(X,V,s),\quad D_{Q^j}\lr (X,V,s;P,Q)=A^j(X,V,s).
\end{align*}
From Assumptions (A1) and (A2), we know that the functional $\lr$ satisfies
\begin{align}\label{lem3_1}
	|\lr(X,V,s;P,Q)|\le C(L)\left(1+\|X\|^2_{\hr_m}+\|V\|^2_{\ur_m}+\|P\|^2_{\hr_m}+\sum_{j=1}^n\left\|Q^j\right\|^2_{\hr_m}\right),
\end{align} 
and for any $(s,P,Q)$, the functional $(X,V)\mapsto \lr(X,V,s;P,Q)\in\br$ is continuously differentiable, with the derivatives
\begin{equation}\label{lem3_4}
	\begin{aligned}
		&\|D_X\lr (X,V,s;P,Q)\|_{\hr_m}+\|D_V\lr(X,V,s;P,Q)\|_{\ur_m}\\
		& \le C(L)\left(1+\|X\|_{\hr_m}+\|V\|_{\ur_m}+\|P\|_{\hr_m}+\sum_{j=1}^n\left\|Q^j\right\|_{\hr_m}\right).
	\end{aligned}
\end{equation}
Let $\hv$ be an optimal control for problem \eqref{intr_3} and $\hx$ be the corresponding state process. We define the adjoint process as the unique solution of the following backward stochastic differential equation (BSDE)
\begin{equation}\label{p}
	\begin{aligned}
		\hat{P}(s)=&D_X G_T\left(\hx(T)\right)+\int_s^T D_X\lr \left(\hx(r),\hv(r),r;\hat{P}(r),\hat{Q}(r)\right)dr\\
		&-\sum_{j=1}^n\int_s^T \hat{Q}^j(r)dw_j(r),\quad s\in[t,T].
	\end{aligned}
\end{equation}
For the solvability of BSDEs in Hilbert spaces, we refer to Remark 5.53 in [\cite{Pardoux}]. From Assumptions (A1) and (A2) and Lemma~\ref{lem:05}, there is a unique solution $\left(\hat{P},\hat{Q}\right)\in L^2_{\mathcal{W}_{tX}}(t,T;\hr_m)\times \left(L^2_{\mathcal{W}_{tX}}(t,T;\hr_m)\right)^n$ of BSDE \eqref{p} satisfying 
\begin{equation}\label{lem06_1}
	\begin{aligned}
		&\sup_{t\le s\le T}\left\|\hat{P}(s)\right\|_{\hr_m}+\sum_{j=1}^n\left\|\hat{Q}^j\right\|_{L^2(t,T;\hr_m)}\\
		\le& C(L,T)(1+ \|X\|_{\hr_m}+\left\|\hv\right\|_{L^2(t,T;\ur_m)}).
	\end{aligned}
\end{equation}
Now we give the necessary condition of control problem \eqref{intr_3}.

\begin{theorem}\label{thm}
	Under Assumptions (A1) and (A2), let $\hv$ be an optimal control for Problem \eqref{intr_3}, $\hx$ be the corresponding controlled state process, and $(\hat{P},\hat{Q})$ be the solution of Equation \eqref{p}. Then, we have the optimality condition
	\begin{align*}
		D_V \lr\left(\hx(s),\hv(s),s;\hat{P}(s),\hat{Q}(s)\right)\overset{\ur_m}{=}0,\quad \text{a.e.}\ s\in[t,T].
	\end{align*}
\end{theorem}

\begin{proof}
	For any $\tv\in L^2_{\mathcal{W}_{tX}}(t,T;\ur_m)$, we set $V^\epsilon:=\hv+\epsilon \tv$ for $\epsilon\in[0,1]$. It is obvious that $V^\epsilon\in L^2_{\mathcal{W}_{tX}}(t,T;\ur_m)$. We denote by $X^\epsilon$ the state process corresponding to the control $V^\epsilon$, and set $Y^\epsilon:=\frac{1}{\epsilon}\left(X^\epsilon-\hx\right)$. From Assumption (A1) and (A2), we have	
	\begin{equation}\label{thm_2}
		\begin{aligned}
			&\frac{1}{\epsilon}\left[J_{tX}(V^\epsilon)-J_{tX}(\hv)\right]\\
			=&\int_t^T \int_0^1 \Big[ \left(D_X G\left(\hx(s)+\lambda\epsilon Y^\epsilon(s),V^\epsilon(s),s\right),Y^\epsilon(s)\right)_{\hr_m} \\
			&\qquad\qquad + \Big(D_V G\left(\hx(s),\hv(s)+\lambda\epsilon\tv(s),s\right),\tv(s)\Big)_{\ur_m} \Big] d\lambda ds\\
			&+\int_0^1 \left( D_X G_T\left(\hx(T)+\lambda\epsilon Y^\epsilon(T)\right),Y^\epsilon(T)\right)_{\hr_m} d\lambda.
		\end{aligned}
	\end{equation}
	Let $D_{\tv}X^{\hv}$ be the solution of equation \eqref{tx} corresponding to $\left(\hv,\tv\right)$. From Assumption (A2), Lemma \ref{lem:05}, and the dominated convergence theorem, we have
	\begin{equation}\label{thm_3}
		\begin{aligned}
			&\lim_{\epsilon\to0}\left[\frac{1}{\epsilon}\left[J_{tX}(V^\epsilon)-J_{tX}(\hv)\right]\right]\\
			=&\int_t^T\Big[\left(D_X G\left(\hx(s),\hv(s),s\right),D_{\tv}X^{\hv}(s)\right)_{\hr_m}\\
			&\quad\qquad+\left(D_V G\left(\hx(s),\hv(s),s\right),\tv(s)\right)_{\ur_m}\Big] ds\\
			&+\left(D_X G_T\left(\hx(T)\right),D_{\tv}X^{\hv}(T)\right)_{\hr_m}.
		\end{aligned}
	\end{equation}
	From \eqref{tx}, \eqref{p}, Assumption (A1) and the definition of the Lagrangian functional $\lr$, we have for $s\in(t,T)$,
	\begin{align*}
		&\frac{d}{ds}\left( \hat{P}(s),D_{\tv}X^{\hv}(s)\right)_{\hr_m} \\
		=&\left(\hat{P}(s),D_X F\left(\hx(s),\hv(s),s\right)\left(D_{\tv}X^{\hv}(s)\right)\right)_{\hr_m}\\
		&+ \left(\hat{P}(s),D_V F\left(\hx(s),\hv(s),s\right)\left(\tv(s)\right)\right)_{\hr_m}\\
		&+\sum_{j=1}^n\bigg[\left(\hat{Q}^j(s),D_X A^j\left(\hx(s),\hv(s),s\right)\left(D_{\tv}X^{\hv}(s)\right)\right)_{\hr_m}\\
		&\qquad\qquad + \left(\hat{Q}^j(s),D_V A^j\left(\hx(s),\hv(s),s\right)\left(\tv(s)\right)\right)_{\hr_m} \bigg]\\
		&-\left(D_X\lr \left(\hx(s),\hv(s),s;\hat{P}(s),\hat{Q}(s)\right),D_{\tv}X^{\hv}(s)\right)_{\hr_m}  \\
		=&\left(\hat{P}(s),D_V F\left(\hx(s),\hv(s),s\right)\left(\tv(s)\right)\right)_{\hr_m}\\
		&+\sum_{j=1}^n \left(\hat{Q}^j(s),D_V A^j\left(\hx(s),\hv(s),s\right)\left(\tv(s)\right)\right)_{\hr_m}\\
		&-\left(D_X G\left(\hx(s),\hv(s),s\right),D_{\tv}X^{\hv}(s)\right)_{\hr_m}.
	\end{align*}
	We integrate against $s$ from $t$ to $T$ so as to obtain
	\begin{equation}\label{thm_5}
		\begin{aligned}
			&\left(D_X G_T\left(\hx(T)\right),D_{\tv}X^{\hv}(T)\right)_{\hr_m}\\
			&+\int_t^T  \left(D_X G\left(\hx(s),\hv(s),s\right),D_{\tv}X^{\hv}(s)\right)_{\hr_m} ds\\
			=&\int_t^T \bigg[\left(\hat{P}(s),D_V F\left(\hx(s),\hv(s),s\right)\left(\tv(s)\right)\right)_{\hr_m}\\
			&\quad\qquad +\sum_{j=1}^n \left(\hat{Q}^j(s),D_V A^j\left(\hx(s),\hv(s),s\right)\left(\tv(s)\right)\right)_{\hr_m} \bigg] ds.
		\end{aligned}
	\end{equation}
	Plugging \eqref{thm_5} into \eqref{thm_3}, from Assumption (A1), we have
	\begin{equation}\label{thm_6}
		\begin{aligned}
			&\lim_{\epsilon\to0}\left[\frac{1}{\epsilon}\left[J_{tX}(V^\epsilon)-J_{tX}(\hv)\right]\right]\\
			=&\int_t^T\bigg[\left(\hat{P}(s),D_V F\left(\hx(s),\hv(s),s\right)\left(\tv(s)\right)\right)_{\hr_m}\\
			&\quad\qquad +\sum_{j=1}^n \left(\hat{Q}^j(s),D_V A^j\left(\hx(s),\hv(s),s\right)\left(\tv(s)\right)\right)_{\hr_m}\\
			&\quad\qquad +\left(D_V G\left(\hx(s),\hv(s),s\right),\tv(s)\right)_{\ur_m}\bigg]ds\\
			=&\int_t^T \left(D_V\lr\left(\hx(s),\hv(s),s;\hat{P}(s),\hat{Q}(s)\right),\tv(s)\right)_{\ur_m} ds.
		\end{aligned}
	\end{equation}
	The necessary condition then follows.   
\end{proof}

\section{Sufficient Condition for Control Problem \eqref{intr_3} on Hilbert Spaces}\label{sec:FB_Hilbert}

In this section, we give the sufficient condition of the control problem \eqref{intr_3}. We need the following assumptions, which are the regularity-enriched version of Assumptions (A1) and (A2).\\
\textbf{(A1')} (i) The maps $F$ and $A^j$ (for $1\le j\le n$) are linear in $V$. That is,
\begin{align*}
	&F(X,V,s)=F_1(X,s)+\mathcal{F}_2(s)V,\\
	&A^j(X,V,s)=A^j_1(X,s)+\mathcal{A}^j_2(s)V,
\end{align*}
where $F_1$ and $A^j_1$ also satisfy Assumption (A1), and $\mathcal{F}_2(s)$ and $\mathcal{A}^j_2(s)$ are linear maps on $\ur_m$, and
\begin{align*}
	&\|\mathcal{F}_2(s)\|_{L(\ur_m;\hr_m)}\le L, \quad \|\mathcal{A}^j_2(s)\|_{L(\ur_m;\hr_m)}\le L.
\end{align*}
(ii) Furthermore, the maps $F_1$ and $A^j_1$ (for $1\le j\le n$) are linear in $X$. That is,
\begin{align*}
	&F_1(X,s)=F_0(s)+\mathcal{F}_1(s)X,\\
	&A^j_1(X,s)=A^j_0(s)+\mathcal{A}^j_1(s)X,\quad 1\le j\le n,
\end{align*}
where $\mathcal{F}_1(s)$ and $\mathcal{A}^j_1(s)$ are linear maps on $\hr_m$, and
\begin{align*}
	&\|F_0(s)\|_{\hr_m}\le L,\quad \|\mathcal{F}_1(s)\|_{L(\hr_m;\hr_m)}\le L,\\ 
	&\|A^j_0(s)\|_{\hr_m}\le L,\quad \|\mathcal{A}^j_1(s)\|_{L(\hr_m;\hr_m)}\le L.
\end{align*}
\textbf{(A2')} The functionals $G$ and $G_T$ satisfy (A2). The maps $D_X G$ and $D_V G$ are $L$-Lipschitz continuous in $(X,V)\in\hr_m\times\ur_m$; the map $D_X G_T$ is $L$-Lipschitz continuous in $X\in\hr_m$.\\
We also need the following convexity assumption.\\
\textbf{(A3)} (i) There exists $\lambda>0$ such that, for any $X\in \hr_m$, $V,V'\in \ur_m$ and $s\in[0,T]$,
\begin{align*}
	&G(X,V',s)-G(X,V,s) \geq  \left(D_V G(X,V,s),V'-V\right)_{\ur_m}+\lambda\|V'-V\|^2_{\ur_m}.
\end{align*}		
(ii) There exists $\lambda> 0$ such that, for any $X,X'\in \hr_m$, $V,V'\in \ur_m$ and $s\in[0,T]$,
\begin{align*}
	&G(X',V',s)-G(X,V,s) \geq  \left(D_X G(X,V,s),X'-X\right)_{\hr_m}+\left(D_V G(X,V,s),V'-V\right)_{\ur_m}\\
	&\qquad\qquad\qquad\qquad\qquad\qquad +\lambda\|V'-V\|^2_{\ur_m},\\
	&\left(D_X G_T(X')-D_X G_T(X),X'-X\right)_{\hr_m}\geq 0.
\end{align*}
We refer to Remark 3.1 in [\cite{AB5}] for the relation between the aforementioned convexity assumption and the displacement monotonicity condition and Lasry-Lions monotonicity condition. We now give a sufficient condition for the control problem \eqref{intr_3}.

\begin{theorem}\label{thm6}
	Under Assumptions (A1')-(i), (A2') and (A3)-(i), there exists a constant $C(L,T)$ depending only on $(L,T)$, such that when $\lambda\geq C(L,T)$, Problem \eqref{intr_3} has a unique optimal control. Furthermore, if Assumptions (A1')-(ii) and (A3)-(ii) are satisfied, it is sufficient to have the well-posedness for any $\lambda> 0$.
\end{theorem}

\begin{proof}	
	From Assumption (A1') we have
	\begin{equation}\label{hv2}
		\begin{aligned}
			&D_V \lr(X,V,s;P,Q)=\left(\mathcal{F}_2(s)\right)^*P+\sum_{j=1}^n \left(\mathcal{A}^j_2(s)\right)^*Q^j+D_V G(X,V,s).
		\end{aligned}
	\end{equation}
	From \eqref{thm_6} and \eqref{hv2}, we have the differentiability of $J_{tX}$: for any $V,\tv\in L_{\mathcal{W}_{tX}}^2(t,T;\ur_m)$, 
	\begin{equation}\label{thm6_1}
		\begin{aligned}
			&\frac{d}{d\epsilon}J_{tX}\left(V+\epsilon\tv\right)\bigg|_{\epsilon=0}\\
			=&\int_t^T \left(D_V \lr(X^V(s),V(s),s;P^V(s),Q^V(s)),\tv(s)\right)_{\ur_m} ds\\
			=&\int_t^T\bigg[\left(\mathcal{F}_2(s)\tv(s),P^V(s)\right)_{\hr_m}+\sum_{j=1}^n \left(\mathcal{A}^j_2(s)\tv(s),Q^{V,j}(s)\right)_{\hr_m}\\
			&\quad\qquad +\left(D_V G(X^V(s),V(s),s),\tv(s)\right)_{\ur_m} \bigg] ds,
		\end{aligned}
	\end{equation}
	where $X^V$ is the state process and $\left(P^V,Q^V\right)$ is the adjoint process corresponding to the control $V$. Now, we prove that $J_{tX}$ is strictly convex when $\lambda$ is large enough. For $V_1,V_2\in  L_{\mathcal{W}_{tX}}^2(t,T;\ur_m)$ and $\theta \in[0,1]$, we can write
	\begin{equation}\label{thm6_2}
		\begin{aligned}
			J_{tX}(\theta V_1+(1-\theta)V_2)&=J_{tX}(V_1+(1-\theta)(V_2-V_1))\\
			&=J_{tX}(V_1)+\int_0^1\frac{d}{d\epsilon}J_{tX}(V_1+\epsilon(1-\theta)(V_2-V_1)) d\epsilon.
		\end{aligned}
	\end{equation}
    (From here on, we use the same $\epsilon$ for both derivative variable and dummy integration variable for the sake of convenience.) From \eqref{thm6_1}, we have
    \small
	\begin{align*}
		&\int_0^1\frac{d}{d\epsilon}J_{tX}(V_1+\epsilon(1-\theta)(V_2-V_1))d\epsilon\\
		=&(1-\theta)\int_0^1\int_t^T \bigg[\left(\mathcal{F}_2(s)(V_2(s)-V_1(s)),P^{1,\theta,\epsilon}(s)\right)_{\hr_m}\\
		&\quad\qquad\qquad\qquad  +\sum_{j=1}^n \left(\mathcal{A}^j_2(s)(V_2(s)-V_1(s)),Q^{1,\theta,\epsilon,j}(s)\right)_{\hr_m}\\
		&\quad\qquad\qquad\qquad +\Big(D_V G\left(X^{1,\theta,\epsilon}(s),V_1(s)+\epsilon(1-\theta)(V_2(s)-V_1(s)),s\right),\  V_2(s)-V_1(s)\Big)_{\ur_m} \bigg]ds d\epsilon,
	\end{align*}
    \normalsize
	where we denote by
	\begin{align*}
		&X^{1,\theta,\epsilon}(\cdot):=X^{V_1+\epsilon(1-\theta)(V_2-V_1)}(\cdot),\quad P^{1,\theta,\epsilon}(\cdot):=P^{V_1+\epsilon(1-\theta)(V_2-V_1)}(\cdot),\\
		&Q^{1,\theta,\epsilon,j}(\cdot):=Q^{V_1+\epsilon(1-\theta)(V_2-V_1),j}(\cdot),\quad 1\le j\le n.
	\end{align*}   
	Similarly, we can also write
	\begin{equation}\label{thm6_3}
		\begin{aligned}
			J_{tX}(\theta V_1+(1-\theta)V_2)&=J_{tX}(V_2+\theta(V_1-V_2))\\
			&=J_{tX}(V_2)+\int_0^1\frac{d}{d\epsilon}J_{tX}(V_2+\epsilon\theta(V_1-V_2)) d\epsilon,
		\end{aligned}
	\end{equation}
	while the second integral term can be expressed as:
	\begin{align*}
		&\int_0^1\frac{d}{d\epsilon}J_{tX}(V_2+\epsilon
		\theta(V_1-V_2))d\epsilon\\
		=&\theta\int_0^1\int_t^T \bigg[\left(\mathcal{F}_2(s)(V_1(s)-V_2(s)),P^{2,\theta,\epsilon}(s)\right)_{\hr_m} \\
		&\quad\qquad\qquad +\sum_{j=1}^n \left(\mathcal{A}^j_2(s)(V_1(s)-V_2(s)),Q^{2,\theta,\epsilon,j}(s)\right)_{\hr_m} \\
		&\quad\qquad\qquad+\left(D_V G(X^{2,\theta,\epsilon}(s),V_2(s)+\epsilon\theta(V_1(s)-V_2(s)),s),\ V_1(s)-V_2(s)\right)_{\ur_m} \bigg]dsd\epsilon,
	\end{align*}
	where we also denote by
	\begin{align*}
		&X^{2,\theta,\epsilon}(\cdot):=X^{V_2+\epsilon\theta(V_1-V_2)}(\cdot),\quad P^{2,\theta,\epsilon}(\cdot):=P^{V_2+\epsilon\theta(V_1-V_2)}(\cdot),\\
		&Q^{2,\theta,\epsilon,j}(\cdot):=Q^{V_2+\epsilon\theta(V_1-V_2),j}(\cdot),\quad 1\le j\le n.
	\end{align*}   
	Adding $\theta$ of \eqref{thm6_2} to $(1-\theta)$ of \eqref{thm6_3}, we have
	\footnotesize
	\begin{equation}\label{thm6_4}
		\begin{aligned}
			&J_{tX}(\theta V_1+(1-\theta)V_2)-\theta J_{tX}(V_1)-(1-\theta)J_{tX}(V_2)\\
			=&\theta(1-\theta)\int_0^1\int_t^T \bigg[\left(\mathcal{F}_2(s)(V_2(s)-V_1(s)),P^{1,\theta,\epsilon}(s)-P^{2,\theta,\epsilon}(s)\right)_{\hr_m}\\
			&\qquad\qquad\qquad\qquad +\sum_{j=1}^n \left(\mathcal{A}^j_2(s)(V_2(s)-V_1(s)),Q^{1,\theta,\epsilon,j}(s)-Q^{2,\theta,\epsilon,j}(s)\right)_{\hr_m}\\
			&\qquad\qquad\qquad\qquad+\Big(D_V G\left(X^{1,\theta,\epsilon}(s),V_1(s)+\epsilon(1-\theta)(V_2(s)-V_1(s)),s\right)\\
			&\qquad\qquad\qquad\qquad\qquad -D_V G\left(X^{2,\theta,\epsilon}(s),V_2(s)+\epsilon\theta(V_1(s)-V_2(s)),s\right),\  V_2(s)-V_1(s)\Big)_{\ur_m}\bigg]dsd\epsilon.
		\end{aligned}
	\end{equation}   
    \normalsize
	Similar to \eqref{lem02_1}, we have the following estimates for $X^{2,\theta,\epsilon}(s)-X^{1,\theta,\epsilon}(s)$,
	\begin{align}\label{thm6_5}
		\sup_{t\le s\le T}\left\|X^{2,\theta,\epsilon}(s)-X^{1,\theta,\epsilon}(s)\right\|_{\hr_m}\le C(L,T)(1-\epsilon)\|V_2(\cdot)-V_1(\cdot)\|_{L^2([t,T];\ur_m)}.
	\end{align}
	Similar to \eqref{lem06_1}, from Assumptions (A2') and the estimate \eqref{thm6_5}, we have the following estimate
	\begin{equation}\label{thm6_6}
		\begin{aligned}
			&\sup_{t\le s\le T}\left\|P^{1,\theta,\epsilon}(s)-P^{2,\theta,\epsilon}(s)\right\|_{\hr_m}+\sum_{j=1}^n \left\| Q^{1,\theta,\epsilon,j}(\cdot)-Q^{2,\theta,\epsilon,j}(\cdot) \right\|_{L^2(t,T;\hr_m)}\\
			\le & C(L,T)(1-\epsilon)\|V_2(\cdot)-V_1(\cdot)\|_{L^2([t,T];\ur_m)}.
		\end{aligned}
	\end{equation}	
	From \eqref{thm6_6} and Assumption (A3), we have
	\begin{equation}\label{thm6_7}
		\begin{aligned}
			&\bigg|\int_0^1\int_t^T\bigg[ \Big(\mathcal{F}_2(s)(V_2(s)-V_1(s)),P^{1,\theta,\epsilon}(s)-P^{2,\theta,\epsilon}(s)\Big)_{\hr_m}\\
			&\quad\qquad\qquad +\sum_{j=1}^n \Big(\mathcal{A}^j_2(s)(V_2(s)-V_1(s)),Q^{1,\theta,\epsilon,j}(s)-Q^{2,\theta,\epsilon,j}(s)\Big)_{\hr_m} \bigg]dsd\epsilon\bigg|\\
			\le& C(L,T)\|V_2(\cdot)-V_1(\cdot)\|^2_{L^2(t,T;\ur_m)}.
		\end{aligned}
	\end{equation}
	From \eqref{thm6_5} and Assumption (A3), we have
	\footnotesize
	\begin{equation}\label{thm6_8}
		\begin{aligned}
			&\int_0^1\int_t^T \bigg(D_V G\left(X^{1,\theta,\epsilon}(s),V_1(s)+\epsilon(1-\theta)(V_2(s)-V_1(s)),s\right)\\
			&\quad\qquad\qquad -D_V G\left(X^{2,\theta,\epsilon}(s),V_2(s)+\epsilon\theta(V_1(s)-V_2(s)),s\right),\  V_2(s)-V_1(s)\bigg)_{\ur_m} dsd\epsilon\\
			\le &\int_0^1\int_t^T \bigg(D_V G \left(X^{2,\theta,\epsilon}(s),V_1(s)+\epsilon(1-\theta)(V_2(s)-V_1(s)),s\right)\\
			&\quad\qquad\qquad-D_V G \left(X^{2,\theta,\epsilon}(s),V_2(s)+\epsilon\theta(V_1(s)-V_2(s)),s\right),\  V_2(s)-V_1(s)\bigg)_{\ur_m} dsd\epsilon\\
			&+\bigg|\int_0^1\int_t^T \bigg(D_V G \left(X^{1,\theta,\epsilon}(s),V_1(s)+\epsilon(1-\theta)(V_2(s)-V_1(s)),s\right)\\
			&\qquad\qquad\qquad -D_V G\left(X^{2,\theta,\epsilon}(s),V_1(s)+\epsilon(1-\theta)(V_2(s)-V_1(s)),s\right),\  V_2(s)-V_1(s)\bigg)_{\ur_m} dsd\epsilon\bigg|\\
			\le& [-2\lambda+C(L,T)]\|V_2(\cdot)-V_1(\cdot)\|^2_{L^2(t,T;\ur_m)}.
		\end{aligned}
	\end{equation}
	\normalsize
	Substituting \eqref{thm6_7} and \eqref{thm6_8} back to \eqref{thm6_4}, we have
	\begin{align*}
		&J_{tX}(\theta V_1+(1-\theta)V_2)-\theta J_{tX}(V_1)-(1-\theta)J_{tX}(V_2)\\
		\le & \theta(1-\theta)[-2\lambda+C(L,T)]\|V_2(\cdot)-V_1(\cdot)\|^2_{L^2(t,T;\ur_m)}.
	\end{align*}   
	Therefore, there exists a constant $C(L,T)$ depending only on $(L,T)$, such that when $\lambda\geq C(L,T)$, $J_{tX}$ is strictly convex. Next, we prove that $J_{tX}(V)\to+\infty$ as $\|V\|_{L^2(t,T;\ur_m)}\to+\infty$. For any $V\in L^2(t,T;\ur_m)$, from Assumptions (A2') and (A3), we have    
	\begin{align*}
		J_{tX}(V)&=\int_t^T G\left(X^V(s),V(s),s\right)ds+G_T\left(X^V(T)\right)\\
		&\geq \int_t^T\Big[G\left(X^V(s),0,s\right)+\left(D_V G\left(X^V(s),0,s\right),V(s)\right)_{\ur_m}+\lambda\|V(s)\|_{\ur_m}^2\Big]ds\\
		&\quad+G_T\left(X^V(T)\right)\\
		&\geq \int_t^T\bigg[ -L\left(1+\left\|X^V(s)\right\|_{\hr_m}^2\right)-L\left(1+\left\|X^V(s)\right\|_{\hr_m}\right)\|V(s)\|_{\ur_m}\\
		&\qquad\qquad +\lambda\|V(s)\|_{\ur_m}^2\bigg]ds-L\left(1+\left\|X^V(T)\right\|_{\hr_m}^2\right).
	\end{align*}
	From \eqref{lem02_1}, we have
	\begin{align*}
		&J_{tX}(V)\geq \big(-\lambda+C(L,T)\big)\|V\|_{L^2(t,T;\ur_m)}^2-C(L,T)\left(1+\|X\|_{\hr_m}^2\right).
	\end{align*}
	So when $\lambda$ is large enough such that $-\lambda+C(L,T)<0$, we know that $J_{tX}(V)\to+\infty$ as $\|V\|_{L^2(t,T;\ur_m)}\to+\infty$. This coercive property and the strict convexity imply that the functional $J$ should possess a unique minimum. That is, \eqref{intr_3} has a unique optimal control.	
	
	Moreover, if Assumption (A1')-(ii) and Assumption (A3)-(ii) are satisfied, then
	\begin{equation}\label{thm6'_1}
		\begin{aligned}
			&J_{tX}(V_2)-J_{tX}(V_1)\\
			=&\int_t^T \left[G \left(X^{V_2}(s),V_2(s),s\right)-G\left(X^{V_1}(s),V_1(s),s\right) \right]ds\\
			&+\left[G_T\left(X^{V_2}(T)\right)-G_T\left(X^{V_1}(T)\right)\right]\\
			\geq&\int_t^T \Big[\left(D_X G(X^{V_1}(s),V_1(s),s),X^{V_2}(s)-X^{V_1}(s)\right)_{\hr_m}\\
			&\qquad +\left(D_V G(X^{V_1}(s),V_1(s),s),V_2(s)-V_1(s)\right)_{\ur_m}+\lambda  \|V_2(s)-V_1(s)\|^2_{\ur_m}\Big]ds\\
			&+\left(D_X G_T(X^{V_1}(T)),X^{V_2}(T)-X^{V_1}(T)\right)_{\hr_m}.
		\end{aligned}
	\end{equation}	    
    Similar to \eqref{thm_5}, from Assumption (A1')-(ii), we have    
    \begin{equation}\label{thm6'_2}
    	\begin{aligned}
    		&\left(D_X G_T(X^{V_1}(T)),X^{V_2}(T)-X^{V_1}(T)\right)_{\hr_m}\\
    		=&\int_t^T \bigg[\left({P}^{V_1}(s),\mathcal{F}_2(s)(V_2(s)-V_1(s))\right)_{\hr_m}\\
    		&\quad\qquad+\sum_{j=1}^n \left({Q^{V_1}}^{j}(s),\mathcal{A}^j_2(s) \left(V_2(s)-V_1(s)\right)\right)_{\hr_m}\\
    		&\quad\qquad -\left(D_X G\left(X^{V_1}(s),V_1(s),s\right),X^{V_2}(s)-X^{V_1}(s)\right)_{\hr_m}\bigg]ds.
    	\end{aligned}    	
    \end{equation}   
    From \eqref{thm6'_1}, \eqref{thm6'_2} and \eqref{thm6_1}, we deduce that    
    \begin{align*}
    	&J_{tX}(V_2)-J_{tX}(V_1)\\
    	\geq&\int_t^T \Big[\left({P}^{V_1}(s),\mathcal{F}_2(s)(V_2(s)-V_1(s))\right)_{\hr_m}+\sum_{j=1}^n \left({Q^{V_1}}^{j}(s),\mathcal{A}^j_2(s) \left(V_2(s)-V_1(s)\right)\right)_{\hr_m}\\
    	&\quad\qquad +\left(D_V G(X^{V_1}(s),V_1(s),s),V_2(s)-V_1(s)\right)_{\ur_m}+\lambda  \|V_2(s)-V_1(s)\|^2_{\ur_m}\Big]ds\\
    	=& \frac{d}{d\epsilon}J_{tX}\left(V_1+\epsilon(V_2-V_1)\right)\bigg|_{\epsilon=0}+\lambda\int_t^T \|V_2(s)-V_1(s)\|^2_{\ur_m} ds.
    \end{align*}	
    That is, $J_{tX}$ is strictly convex when $\lambda>0$. Therefore, \eqref{intr_3} has a unique optimal control for any $\lambda>0$.
\end{proof}

We derive from Theorem~\ref{thm} the following FBSDEs for $(Y_{Xt}(s),P_{Xt}(s),Q_{Xt}(s))$ for $s\in[t,T]$,
\begin{equation}\label{FB:2}
	\left\{
	\begin{aligned}
		&Y_{Xt}(s)=X+\int_t^s F(Y_{Xt}(r),U_{Xt}(r),r)dr+\sum_{j=1}^n\int_t^s  A^j(Y_{Xt}(r),U_{Xt}(r),r)dw_j(r),\\
		&P_{Xt}(s)=D_X G_T(Y_{Xt}(T))+\int_s^T D_X \lr\left(Y_{Xt}(r),U_{Xt}(r),r;P_{Xt}(r),Q_{Xt}(r)\right)dr\\
		&\qquad\qquad -\sum_{j=1}^n \int_s^T Q_{Xt}^j(r)dw_j(r),
	\end{aligned}
	\right.
\end{equation}
with $U_{Xt}(s)$ satisfying the following optimal condition
\begin{align}\label{U}
	D_V\lr(Y_{Xt}(s),U_{Xt}(s),s;P_{Xt}(s),Q_{Xt}(s)){=}0,\quad \text{a.e.}\ s\in[t,T].
\end{align}
As a consequence of Theorems~\ref{thm} and \ref{thm6}, we have the solvability of FBSDEs \eqref{FB:2}-\eqref{U} in space $L^2_{\mathcal{W}_{tX}}(t,T;\hr_m)\times L^2_{\mathcal{W}_{tX}}(t,T;\hr_m)\times (L^2_{\mathcal{W}_{tX}}(t,T;\hr_m))^n$ under Assumptions (A1'), (A2') and (A3).

\section{Application to Mean Field Type Control Problems}\label{sec:MFC}

In this section, we apply our results in Section~\ref{sec:Hilbert} back to the mean field type control problem \eqref{eq:3-21}-\eqref{eq:3-22} for $(t,m)\in [0,T]\times\pr_2(\br^n)$ with an initial $X\in \hr_m$ (being independent of $\mathcal{W}_t$). For any fixed $m\in\pr_2(\brn)$, we can make a connection between the mean field type control problem \eqref{eq:3-21}-\eqref{eq:3-22} and the control problem \eqref{intr_3} in Hilbert spaces by setting for $(X,V,s)\in \hr_m\times \ur_m\times[0,T]$ and $x\in\brn$,
\begin{align}
	&\left.F(X,V,s)\right|_x:=f(X_x,X\otimes m,V_x,s),\label{G}\\
	&\left.A^j(X,V,s)\right|_x:=\sigma^j(X_x,X\otimes m,V_x,s),\quad 1\le j\le n, \label{A}\\
	&G(X,V,s):=\e\left[\int_\brn g(X_x,X\otimes m,V_x,s)dm(x)\right],\label{F}\\
	&G_T(X):=\e\left[\int_\brn g_T(X_x,X\otimes m)dm(x)\right].\label{F_T}
\end{align}
We define the Lagrangian $L:\brn\times\pr_2(\brn)\times\brd\times[0,T]\times\brn\times\br^{n\times n}\to\br$ as
\begin{align}\label{H'}
	L(x,m,v,s;p,q):=p^* f(x,v,m,s)+\sum_{j=1}^n {q^{j}}^* \sigma^j(x,v,m,s)+g(x,v,m,s),
\end{align}
and define $\lr:\hr_m\times \ur_m\times[0,T]\times \hr_m\times \hr_m^n\to\br$ as
\begin{align}\label{H''}
	\lr(X,V,s;P,Q):=\e\left[\int_\brn {L}(X_x,X\otimes m,V_x,s;P_x,Q_x)dm(x)\right].
\end{align}
Then, we know from \eqref{G}-\eqref{H''} that $(G,A,F,\lr)$ satisfies \eqref{H}.

We first give the connection between the linear functional derivative in $\pr_2(\brn)$ and the Gateaux derivative in $\hr_m$. Let $k:\pr_2(\brn)\to\br$ be a differentiable functional such that for any $\mu\in\pr_2(\brn)$, the map $\brn\ni\xi\mapsto\frac{d k}{d\nu}(\mu)(\xi)\in\br$ is differentiable with the derivative $D_\xi\frac{d k}{d\nu}(\mu)(\xi)$ being continuous in $(\mu,\xi)$ and
\begin{align*}
	\left|D_\xi\frac{d k}{d\nu}(\mu)(\xi)\right|\le c(\mu)(1+|\xi|),\quad (\mu,\xi)\in\pr_2(\brn)\times\brn.
\end{align*}
For $m\in\pr_2(\brn)$, we define $K:\hr_m\to\br$ as
\begin{align*}
	K(X):=k(X\otimes m), \quad X\in \hr_m.
\end{align*}
Then, $K$ is G\^ateaux differentiable, and
\begin{align}\label{lem01_1}
	\left.D_X K(X)\right|_x=D_\xi\frac{d k}{d\nu}(X\otimes m)(X_x),\quad X\in \hr_m.
\end{align}
By letting $X$ to be the identity function $I$, \eqref{lem01_1} becomes
\begin{align*}
	\left.D_X K(I)\right|_x=D_x\frac{d k}{d\nu}(m)(x),
\end{align*}
which is identical to the $L$-derivative $\dd_m k(m)(x)$ as introduced in Carmona-Delarue [\cite{book_mfg}]. We also refer to Bensoussan-Frehse-Yam [\cite{AB4}] for further discussion about its connection with Wasserstein gradients. We then give the corresponding differentiability of $F$, $A^j$, $G$, $G_T$ and $\lr$ in $(X,V)\in \hr_m\times \ur_m$ under the following assumptions. For notational convenience, we use the same constant $l>0$ for all the conditions below. \\
\textbf{(B1)} The functions $f$ and $\sigma^j$ (for $1\le j\le n$) satisfy for $(x,m,v,s)\in \brn\times \pr_2(\brn)\times \brd\times[0,T]$,
\begin{align*}
	&|f(x,m,v,s)|\le l\left[1+|x|+W_2(m,\delta_0)+|v|\right],\\
	&|\sigma^j(x,m,v,s)|\le l\left[1+|x|+W_2(m,\delta_0)+|v|\right],
\end{align*}
and they are differentiable in $(x,m,v)\in\brn\times\pr_2(\br^n)\times\brd$. The derivatives $\frac{d f}{d \nu}(x,m,v,s)(\xi)$ and $\frac{d \sigma^j}{d \nu}(x,m,v,s)(\xi)$ are differentiable in $\xi$. The derivatives $\left(D_xf,D_\xi \frac{d f}{d \nu},D_vf,D_x\sigma^j,D_\xi \frac{d \sigma^j}{d \nu},D_v\sigma^j\right)$ are bounded by $l$ and continuous in all arguments.\\
\textbf{(B2)} The functions $g$ and $g_T$ satisfy for $(x,m,v,s)\in \brn\times \pr_2(\brn)\times \brd\times[0,T]$,
\begin{align*}
	&|g(x,m,v,s)|\le l\left(1+|x|^2+W^2_2(m,\delta_0)+|v|^2\right),\\  
	&|g_T(x,m)|\le l\left(1+|x|^2+W^2_2(m,\delta_0)\right).
\end{align*}
The running cost function $g$ is differentiable in $(x,m,v)\in\brn\times\pr_2(\br^n)\times\brd$, and the derivative $\frac{d g}{d \nu}(x,m,v,s)(\xi)$ is differentiable in $\xi$. The terminal cost function $g_T$ is differentiable in $(x,m)\in\brn\times\pr_2(\br^n)$, and the derivative $\frac{d g_T}{d \nu}(x,m)(\xi)$ is differentiable in $\xi$. The derivatives $\left(D_x g,D_\xi\frac{d g}{d \nu},D_v g,D_x g_T,D_\xi\frac{d g_T}{d \nu}\right)$ are continuous in all of their own arguments, and for $(x,m,v,s)\in \brn\times \brd\times \pr_2(\brn)\times[0,T]$,
\begin{align*}
	&|(D_x g,D_v g)(x,m,v,s)|\le l\left(1+|x|+W_2(m,\delta_0)+|v|\right),\\
	&\left|D_\xi\frac{d g}{d \nu}(x,m,v,s)(\xi)\right|\le l\left(1+|x|+W_2(m,\delta_0)+|v|+|\xi|\right),\\
	&|D_x g_T(x,m)|\le l\left(1+|x|+W_2(m,\delta_0)\right),\\ 
	&\left|D_\xi\frac{d g_T}{d \nu}(x,m)(\xi)\right|\le l\left(1+|x|+W_2(m,\delta_0)+|\xi|\right).
\end{align*}
In the rest of this article, for any random variable $\xi$, we use $\bar{\xi}$ to denote its independent copy, and use $\bar{\e}[\bar{\xi}]$ to denote its expectation; at the end, at the most we may only need countably many of independent copies, one can simply construct them right away in $(\Omega,\mathcal{A},\mathbb{P})$ before the discussion of our control problem. The following lemma is proven in \ref{pf:lem:1}.

\begin{lemma}\label{lem:1}
	Under Assumption (B1) and (B2), $F$ as defined in \eqref{G} and $A^j$ as defined in \eqref{A} satisfy (A1); $G$ as defined in \eqref{F} and $G_T$ as defined in \eqref{F_T} also satisfy (A2), with a constant $C(l)$ depending only on $l$. The functional $\lr$ defined in \eqref{H''} satisfies \eqref{lem3_1}. For any $(s,P,Q)\in [0,T]\times\hr_m\times\hr_m^n$, the functional $\hr_m\times \ur_m\ni(X,V)\mapsto \lr(X,V,s;P,Q)\in\br$ is continuously differentiable, with the derivatives satisfying \eqref{lem3_4}. Moreover, for $s\in[0,T]$, $X,\tx\in \hr_m$, $V,\tv\in \ur_m$, we have for $x\in\brn$,
	\footnotesize
	\begin{align}
		&\left.D_X F(X,V,s)\left(\tx\right)\right|_x=\left(D_x f(X_x,X\otimes m,V_x,s)\right)^*\tx_x \notag\\
		&\quad\qquad\qquad\qquad\qquad\qquad +\bar{\e}\left[\int_\brn \left(D_\xi\frac{d f}{d\nu}(X_x,X\otimes m,V_x,s)\left(\bar{X}_y\right)\right)^*\bar{\tx}_y dm(y)\right];\label{lem1_2}\\
		&\left.D_V F(X,V,s)\left(\tv\right)\right|_x=\left(D_v f(X_x,X\otimes m,V_x,s)\right)^*\tv_x;\label{lem1_3}\\
		&\left.D_X A^j(X,V,s)\left(\tx\right)\right|_x=\left(D_x \sigma^j(X_x,X\otimes m,V_x,s)\right)^*\tx_x\notag\\
		&\qquad\qquad\qquad\qquad\qquad\qquad+\bar{\e}\left[\int_\brn \left(D_\xi\frac{d \sigma^j}{d\nu}(X_x,X\otimes m,V_x,s)\left(\bar{X}_y\right)\right)^*\bar{\tx}_y dm(y)\right];\label{lem1_2'}\\
		&\left.D_V A^j(X,V,s)\left(\tv\right)\right|_x=\left(D_v \sigma^j(X_x,X\otimes m,V_x,s)\right)^*\tv_x;\label{lem1_3'}\\
		&\left.D_X G(X,V,s)\right|_x=D_x g(X_x,X\otimes m,V_x,s)\notag\\
		&\quad\qquad\qquad\qquad\qquad+\bar{\e}\left[\int_\brn D_\xi \frac{d g}{d\nu}\left(\bar{X}_y,X\otimes m,\bar{V}_y,s\right)(X_x) dm(y)\right];\label{lem2_2}\\
		&\left.D_V G(X,V,s)\right|_x=D_v g(X_x,X\otimes m,V_x,s);\label{lem2_3}\\
		&\left.D_X G_T(X)\right|_x=D_x h(X_x,X\otimes m)+\bar{\e}\left[\int_\brn D_\xi\frac{d g_T}{d\nu}\left(\bar{X}_y,X\otimes m\right)(X_x) dm(y)\right]; \label{lem2_4}
	\end{align}
    \normalsize
	and for $(P,Q)\in \hr_m\times\hr_m^n$, we have for $x\in\brn$,
	\footnotesize
	\begin{align}	
		\left.D_X\lr(X,V,s;P,Q)\right|_x=&D_x f(X_x,X\otimes m,V_x,s) P_x+\sum_{j=1}^n  \sigma^j(X_x,X\otimes m,V_x,s) Q^j_x \notag\\
		& +D_x g(X_x,X\otimes m,V_x,s)\notag\\
		&+\bar{\e}\left[\int_\brn  D_\xi\frac{d f}{d\nu}\left(\bar{X}_y,X\otimes m,\bar{V}_y,s\right)(X_x) {P}_y dm(y)\right] \notag \\
		&+\sum_{j=1}^n \bar{\e}\left[\int_\brn D_\xi\frac{d \sigma^j}{d\nu}\left(\bar{X}_y,X\otimes m,\bar{V}_y,s\right)(X_x) {{Q}^j_y} dm(y)\right] \notag \\
		& +\bar{\e}\left[\int_\brn D_\xi\frac{d g}{d\nu}\left(\bar{X}_y,X\otimes m,\bar{V}_y,s\right)(X_x) dm(y)\right],\label{lem3_2}\\
		\left.D_V \lr(X,V,s;P,Q)\right|_x=&D_v f(X_x,X\otimes m,V_x,s)P_x+\sum_{j=1}^n  D_v \sigma^j (X_x,X\otimes m,V_x,s) {Q^j_x}\notag\\
		& +D_v g(X_x,X\otimes m,V_x,s), \label{lem3_3}
	\end{align}
    \normalsize
	where $\left(\bar{X},\bar{\tx},\bar{V}\right)$ is an independent copy of $(X,\tx,V)$.
\end{lemma}

Let $\hat{v}_x(s)$ be an optimal control for the mean field control problem \eqref{eq:3-21}-\eqref{eq:3-22} and $\hx_x(s)$ be the corresponding state process. The adjoint process is defined as
\footnotesize
\begin{equation}\label{p'}
	\begin{aligned}
		\hp_x(s)=& D_xg_T\left(\hx_x(T),\hx(T)\otimes m\right)\\
		&+\bar{\e}\left[\int_\brn D_\xi\frac{d g_T}{d\nu}\left(\bar{\hx}_y(T),\hx(T)\otimes m\right)\left(\hx_x(T)\right) dm(y)\right]\\
		&+\int_s^T \bigg[D_x {L}\left(\hx_x(r),\hx(r)\otimes m,\hat{v}_x(r),r;\hp_x(r),\hq_x(r)\right)\\
		&\qquad\qquad+\bar{\e}\bigg(\int_\brn D_\xi \frac{d L}{d\nu}\left(\bar{\hx}_y(r),\hx(r)\otimes m,\bar{\hat{v}}_y(r),r;\bar{\hp}_y(r),\bar{\hq}_y(r)\right) \left(\hx_x(r)\right) dm(y)\bigg)\bigg]dr\\
		&-\sum_{j=1}^n \int_s^T \hq^j_x(r)dw_j(r),\quad s\in[t,T],
	\end{aligned}
\end{equation}
\normalsize
where $\left(\bar{\hat{X}}(s),\bar{\hat{v}}(s),\bar{\hat{P}}(s),\bar{\hat{Q}}(s)\right)$ is an independent copy of $\left(\hx(s),\hat{v}(s),\hp(s),\hq(s)\right)$. By Lemma~\ref{lem:1}, we know that $\left(\hp_x(s),\hq_x(s)\right)$, defined in \eqref{p'}, satisfies Equation \eqref{p}. The following necessary condition is then a direct consequence of Theorem~\ref{thm} and Lemma~\ref{lem:1}.
\begin{theorem}\label{thm'}
	Under Assumptions (B1) and (B2), let $\hat{v}_x(s)$ be an optimal control for the mean field control problem \eqref{eq:3-21}-\eqref{eq:3-22}, $\hx_x(s)$ be the corresponding controlled state process, and $\left(\hp_x(s),\hq_x(s)\right)$ be the corresponding adjoint process. Then, we have the optimality condition
	\small
	\begin{align*}
		D_v L\left(\hx_x(s),\hx(s)\otimes m,\hat{v}_x(s),s;\hp_x(s),\hq_x(s)\right)=0,\  \text{a.e.}\ s\in[t,T],\  \text{a.s.}\ dm(x), \ \text{a.s.}\ d\mathbb{P}(\omega).
	\end{align*}
    \normalsize
\end{theorem}

To give a sufficient condition for the mean field type control problem \eqref{eq:3-21}-\eqref{eq:3-22}, we need the following assumptions, which are the regularity-enriched version of Assumptions (B1) and (B2).\\
\textbf{(B1')} (i) The coefficients $f$ and $\sigma^j$ (for $1\le j\le n$) are linear in $v$. That is,
\begin{align*}
	&f(x,m,v,s)=f_1(x,m,s)+f_2(s)v,\\
	&\sigma^j(x,m,v,s)=\sigma^j_1(x,m,s)+\sigma^j_2(s)v,
\end{align*}
with $f_1,\sigma^j_1$ satisfying Assumption (B1),  and $f_2(s),\sigma^j_2(s)\in\br^{n\times d}$ being bounded by $l$.\\
(ii) Furthermore, the functions $f_1$ and $\sigma_1^j$ (for $1\le j\le n$)  are linear in $x$ and $m$. That is,
\begin{align*}
	&f_1(x,m,s)=f_0(s)+f_1(s)x+\bar{f}_1(s)\int_\brn ydm(y),\\
	&\sigma^j_1(x,m,s)=\sigma^j_0(s)+\sigma^j_1(s)x+\bar{\sigma}^j_1(s)\int_\brn ydm(y),
\end{align*}
with $g_0(s),\sigma^j_0(s)\in \brn$ and $f_1(s),\bar{f}_1(s),\sigma^j_1(s),\bar{\sigma}^j_1(s)\in\br^{n\times n}$ being bounded by $l$.\\
\textbf{(B2')} The functionals $g$ and $g_T$ satisfy (B2). The derivatives 
\begin{align*}
	\left(D_x g,D_\xi \frac{d g}{d\nu},D_v g,D_x g_T,D_\xi\frac{d g_T}{d\nu}\right)
\end{align*}
are $l$-Lipschitz continuous in $(x,m,v,\xi)$. \\
We also need the following convexity assumption.\\
\textbf{(B3)} (i) There exists $\lambda> 0$ such that for any $(x,m,v,v')\in\brn\times\pr_2(\brn)\times\brd\times\brd$ and $s\in[0,T]$,
\begin{align*}
	&g(x,m,v',s)-g(x,m,v,s)\geq \left(D_v g (x,m,v,s)\right)^* (v'-v)+\lambda |v'-v|^2.
\end{align*}
(ii) There exists $\lambda>0 $ such that, for any $x,x',\xi,\xi'\in\brn$, $v,v'\in\brd$ and $s\in[0,T]$,
\begin{align*}
	g(x',m',v',s)-g(x,m,v,s)&\geq \left(D_x g (x,m,v,s)\right)^* (x'-x)\\
	&\qquad +\int_\brn\frac{d g}{d\nu}(x,m,v,s)(\xi)d(m'-m)(\xi)\\
	&\qquad +\left(D_v g (x,m,v,s)\right)^* (v'-v)+\lambda |v'-v|^2;\\
	\frac{dg}{d\nu}(x,m,v,s)(\xi')-\frac{dg}{d\nu}(x,m,v,s)(\xi)&\geq \left(D_\xi \frac{dg}{d\nu}(x,m,v,s)(\xi)\right)^* (\xi'-\xi);\\
	g_T(x',m')-g_T(x,m)&\geq \left(D_x g_T (x,m)\right)^* (x'-x)\\
	&\qquad +\int_\brn\frac{d g_T}{d\nu}(x,m)(\xi)d(m'-m)(\xi);\\
	\frac{dg_T}{d\nu}(x,m)(\xi')-\frac{dg_T}{d\nu}(x,m)(\xi)&\geq \left(D_\xi \frac{dg_T}{d\nu}(x,m)(\xi)\right)^*  (\xi'-\xi).
\end{align*}
From Assumption (B1'), we know that $F$ and $A^j$ defined in \eqref{G} and \eqref{A} satisfy the condition (A1'). From Assumption (B2') and Lemma~\ref{lem:1}, we know that functionals $G$ and $G_T$ respectively defined in \eqref{F} and \eqref{F_T} satisfy the Lipschitz continuity conditions in (A2') with a constant $C(l)$ depending only on $l$. By Assumption (B3)-(i), we have for any $X\in \hr_m$, $V,V'\in \ur_m$, and $s\in[0,T]$, 
\begin{align*}
	&G(X,V',s)-G(X,V,s)\\
	=\ &\e\left[\int_\brn \left[{g}(X_x,X\otimes m,V'_x,s)- {g}(X_x,X\otimes m,V_x,s)\right]dm(x)\right]\\
	\geq\ & \e\left[\int_\brn \left[\left(D_v g(X_x,X\otimes m,V_x,s)\right)^* (V'_x-V_x)+\lambda |V'_x-V_x|^2\right]dm(x)\right]\\
	=\ &\e\left[\int_\brn \left(\left.D_V G(X,V,s)\right|_x\right)^* (V'_x-V_x) dm(x)\right]+\lambda\|V'-V\|^2_{\ur_m}.
\end{align*} 
That is, $G$ satisfies the convexity conditions in (A3)-(i). Furthermore, if Assumption (B3)-(ii) is satisfied, by Fubini's theorem, we have
\small
\begin{align*}
	&G(X',V',s)-G(X',V,s)\\
	=\ &\e\left[\int_\brn \left[{g}(X'_x,X'\otimes m,V'_x,s)- {g}(X_x,X\otimes m,V_x,s)\right]dm(x)\right]\\
	\geq \ &\e\bigg[\int_\brn \bigg[\left(D_x g (X_x,X\otimes m,V_x,s)\right)^* (X'_x-X_x)\\
	&\qquad\qquad +\int_\brn\frac{d g}{d\nu}(X_x,X\otimes m,V_x,s)(\xi)d(X'\otimes m-X\otimes m)(\xi)\\
	&\qquad\qquad +\left(D_v g (X_x,X\otimes m,V_x,s)\right)^* (V'_x-V_x)+\lambda |V'_x-V_x|^2\bigg]dm(x)\bigg]\\
	=\ &\e\bigg[\int_\brn \bigg[ \left(D_x g (X_x,X\otimes m,V_x,s)\right)^* (X'_x-X_x)\\
	&\qquad\qquad +\bar{\e}\bigg(\int_\brn\Big(\frac{d g}{d\nu}(X_x,X\otimes m,V_x,s)(\bar{X}'_\xi)\\
	&\qquad\qquad\qquad\qquad\qquad -\frac{d g}{d\nu}(X_x,X\otimes m,V_x,s)(\bar{X}_\xi)\Big)dm(\xi)\bigg)\bigg] dm(x)\bigg]\\
	&\quad +\left(D_V G(X,V,s),\ V'-V\right)_{\ur_m}+\lambda\|V'-V\|^2_{\ur_m}\\
	\geq\ & \e\bigg[\int_\brn \bigg[\left(D_x g (X_x,X\otimes m,V_x,s)\right)^* (X'_x-X_x)\\
	&\qquad\qquad +\bar{\e}\bigg(\int_\brn \left(D_\xi \frac{d g}{d\nu}(X_x,X\otimes m,V_x,s)(\bar{X}_\xi)\right)^* (\bar{X}'_\xi-\bar{X}_\xi)dm(\xi)\bigg)\bigg]dm(x)\bigg]\\
	&\quad +\left(D_V G(X,V,s),\ V'-V\right)_{\ur_m}+\lambda\left\|V'- V\right\|^2_{\ur_m}\\
	=\ & \e\bigg[\int_\brn \bigg[ D_x g (X_x,X\otimes m,V_x,s)\\
	&\qquad\qquad +\bar{\e}\left(\int_\brn D_\xi \frac{d g}{d\nu}(\bar{X}_\xi,X\otimes m,\bar{V}_\xi,s)(X_x)dm(\xi)\right)\bigg]^* (X'_x-X_x)dm(x)\bigg]\\
	&\quad +\left(D_V G(X,V,s),\ V'-V\right)_{\ur_m}+\lambda\left\|V'- V\right\|^2_{\ur_m}\\
	=\ &\left(D_X G(X,V,s),\ X'-X\right)_{\hr_m}+\left(D_V G(X,V,s),\ V'-V\right)_{\ur_m}+\lambda\left\|V'- V\right\|^2_{\ur_m},
\end{align*} 
\normalsize
where $(\bar{X},\bar{X}',\bar{V})$ is an independent copy of $(X,X',V)$. That is, $G$ satisfies the convexity conditions in (A3)-(ii). The argument leading to the convexity for $G_T$ can be shown similarly. As a consequence of Theorem~\ref{thm6}, we have the following solvability of mean field type control problem \eqref{eq:3-21}-\eqref{eq:3-22}.

\begin{theorem}\label{thm1'}
	Under Assumptions (B1')-(i), (B2') and (B3)-(i), there exists a constant $C(l,T)$ depending only on $(l,T)$, such that when $\lambda\geq C(l,T)$, $J_{tX}$ is strictly convex, and the mean field control problem \eqref{eq:3-21}-\eqref{eq:3-22} has a unique optimal control. Furthermore, if Assumptions (A1')-(ii) and (A3)-(ii) are satisfied, it is sufficient to have the well-posedness for any $\lambda> 0$.
\end{theorem}
We derive from Theorem~\ref{thm'} the following FBSDEs, for $s\in[t,T]$,
\small
\begin{equation}\label{FB:3}
	\left\{
	\begin{aligned}
		&Y_{X_xmt}(s)=X_x+\int_t^r f(Y_{X_xmt}(r),Y_{Xmt}(r)\otimes m,u_{X_xmt}(r),r)dr\\
		&\ \ \qquad\qquad\qquad +\sum_{j=1}^n \int_t^s \sigma^j(Y_{X_xmt}(r),Y_{Xmt}(r)\otimes m,u_{X_xmt}(r),r)dw_j(r),\\
		&P_{X_xmt}(s)=D_xg_T(Y_{X_xmt}(T),Y_{Xmt}(T)\otimes m)\\
		&\quad\qquad\qquad +\bar{\e}\left[\int_\brn D_\xi\frac{d g_T}{d\nu}\left(\bar{Y}_{\bx_ymt}(T),Y_{Xmt}(T)\otimes m\right)(Y_{X_xmt}(T)) dm(y)\right]\\
		&\quad\qquad\qquad +\int_s^T \bigg[D_x {L}(Y_{X_xmt}(r),Y_{Xmt}(r)\otimes m,u_{X_xmt}(r),r;P_{X_xmt}(r),Q_{X_xmt}(r))\\
		&\quad\qquad\qquad\qquad\qquad +\bar{\e}\bigg(\int_\brn D_\xi \frac{d L}{d\nu}\Big(\bar{Y}_{\bx_ymt}(r),Y_{Xmt}(r)\otimes m,\bar{u}_{\bx_ymt}(r),r;\\
		&\quad\qquad\qquad\qquad\qquad\qquad\qquad\qquad\qquad \bar{P}_{\bx_ymt}(r),\bar{Q}_{\bx_ymt}(r)\Big)\ (X_{X_xmt}(r)) dm(y)\bigg)\bigg]dr\\
		&\quad\qquad\qquad -\sum_{j=1}^n \int_s^T Q^j_{X_xmt}(r)dw_j(r),
	\end{aligned}
	\right.
\end{equation}
\normalsize
where $u_{X_x mt}(s)$ satisfies the following optimal condition
\begin{equation}\label{u}
	\begin{aligned}
		D_v L(Y_{X_xmt}(s),Y_{Xmt}(s)\otimes m,u_{X_xmt}(s),s;P_{X_xmt}(s),Q_{X_xmt}(s))=0,\\
		\text{a.e.}\ s\in[t,T],\  \text{a.s.}\ dm(x), \ \text{a.s.}\ d\mathbb{P}(\omega),
	\end{aligned}
\end{equation}
such that the process $(\bar{Y}_{\bx mt}(s),\bar{u}_{\bx mt}(s),\bar{P}_{\bx mt}(s),\bar{Q}_{\bx mt}(s))$ is an independent copy of $(Y_{Xmt}(s),u_{Xmt}(s),P_{Xmt}(s),Q_{Xmt}(s))$. For the Lagrangian $L$, we define $\hat{v}(x,m,s;p,q)$ satisfying the following:
\begin{align}\label{hv'}
	D_v L \left(x,m,\hat{v}(x,m,s;p,q),s;p,q\right)=0,
\end{align}
and then define $\hv:\hr_m \times[0,T]\times\hr_m\times\hr_m^n\to \ur_m$ as
\begin{align}\label{v}
	\left.\hv(X,s;P,Q)\right|_x:=\hat{v}(X_x,X\otimes m,s;P_x,Q_x),\quad x\in\brn.
\end{align}
From \eqref{H''} and \eqref{v}, we know that
\begin{align*}
	\hv(X,s;P,Q)=\argmin_{V\in \ur_m}\lr (X,V,s;P,Q),\quad (X,s,P,Q)\in \hr_m \times[0,T]\times {\hr}_m\times {\hr}_m^n.
\end{align*}
Then, we know that $u_{X_\cdot mt}(\cdot)$ defined in \eqref{u} coincides with $U_{Xt}(\cdot)$ as defined in \eqref{U}. Therefore, this FBSDEs \eqref{FB:3}-\eqref{u} coincide with the FBSEDs \eqref{FB:2}-\eqref{U}. 

\begin{remark}	
	From above, we can see that the optimal control $u_{X_xmt}(s)$ for the mean field control problem \eqref{eq:3-21}-\eqref{eq:3-22} is a feedback one, namely:
	\begin{align*}
		u_{X_xmt}(s)=\hat{v}\left(Y_{X_xmt}(s),Y_{Xmt}(s)\otimes m,s;P_{X_xmt}(s),Q_{X_xmt}(s)\right),
	\end{align*}
	and we can have the regularity for the function $\hat{v}$ in accordance with the conditions of the coefficient functions (we refer to [\cite{AB6},\cite{AB5},\cite{AB9}]). Therefore, we know that the system of FBSDEs \eqref{FB:3}-\eqref{u} has a unique adapted solution. We can see that the monotonicity conditions for the FBSDEs of our mean field control problem are actually the convexity conditions of the coefficient functions. Certainly, it is different from mean field games. Indeed, when we consider a mean field game, the associated FBSDEs should be
	\small
	\begin{equation}\label{FB:5}
		\left\{
		\begin{aligned}
			&Y_{X_xmt}(s)=X_x+\int_t^s f(Y_{X_xmt}(r),Y_{Xmt}(r)\otimes m,u_{X_xmt}(r),r)dr\\
			&\qquad\qquad\qquad +\sum_{j=1}^n\int_t^s \sigma^j(Y_{X_xmt}(r),Y_{Xmt}(r)\otimes m,u_{X_xmt}(r),r)dw_j(r),\\
			&P_{X_xmt}(s)=D_xg_T(Y_{X_xmt}(T),Y_{Xmt}(T)\otimes m)\\
			&\qquad\qquad\qquad +\int_s^T \big[D_x {L}(Y_{X_xmt}(r),Y_{Xmt}(r)\otimes m,u_{X_xmt}(r),r;P_{X_xmt}(r),Q_{X_xmt}(r))\big]dr\\
			&\qquad\qquad\qquad -\sum_{j=1}^n \int_s^T Q^j_{X_xmt}(r)dw_j(r),
		\end{aligned}
		\right.
	\end{equation}
    \normalsize
	with $u_{X_x mt}(s)$ satisfying the optimal condition \eqref{u}. Note that the derivative with respect to the distribution term does not appear. Therefore, the monotonicity conditions for $g$ and $g_T$ are not the convexity conditions anymore. For instance, take $g_T$ as an example. In view of \eqref{FB:5}, the monotonicity condition for $g_T$ in mean field games should be
	\begin{align}\label{mono_7}
		\left(D_x g_T(X_x',X'\otimes m)-D_x g_T(X_x,X\otimes m),\ X'_x-X_x\right)_{\hr_m}\geq 0,\quad X,X'\in\hr_m.
	\end{align}
	Note that
	\begin{align*}
		&\left(D_x g_T(X_x',X'\otimes m)-D_x g_T(X_x,X\otimes m),\ X'_x-X_x\right)_{\hr_m}\\
		=\ &\e\left[\int_\brn \left(\left(D_x g_T\left(X_x',X'\otimes m\right)-D_x g_T\left(X_x,X\otimes m\right)\right)\right)^*  \left(X'_x-X_x\right) dm(x)\right]\\
		=\ &\e\left[ \left(\left(D_x g_T\left(X_\eta',\lr X_\eta'\right)-D_x g_T\left(X_\eta,X_\eta\right)\right)\right)^* \left(X'_\eta-X_\eta\right) \right],
	\end{align*}
	when $\eta$ is equipped with a probability $m\otimes\mathbb{P}|_{[0,t]}$, where $\mathbb{P}|_{[0,t]}$ is generated by $w(s),\ s\in[0,t]$. Therefore, condition \eqref{mono_7} concides with the displacement monotonicity
	\begin{align}\label{mono_8}
		\e\left[\left(\left(D_x g_T\left(\eta^2,\lr \eta^2\right)-D_x g_T\left(\eta^1,\lr \eta^1\right)\right)\right)^*  \left(\eta^2-\eta^1\right)\right]\geq 0,
	\end{align}
	which is proposed by Ahuja [\cite{SA}]. When $h$ is continuously differentiable in $x$ and is also convex in $x$, the displacement monotonicity \eqref{mono_8} is satisfied if the following condition is satisfied
	\begin{align}\label{mono_9}
		\e\left[g_T\left(\eta^1,\lr \eta^1\right)+g_T\left(\eta^2,\lr \eta^2\right)-g_T\left(\eta^1,\lr \eta^2\right)-g_T\left(\eta^2,\lr \eta^1\right)\right]\geq 0.
	\end{align}
	Condition \eqref{mono_9} is called well-known Lasry-Lions monotonicity, extensively used in the literature [\cite{CP},\cite{CDLL},\cite{book_mfg},\cite{CJF}].	For further discussion on monotonicity conditions for mean field control problem and mean field game, we also refer to Remark 3.1 in [\cite{AB5}] and Section 5.3 in [\cite{AB9}].	 
\end{remark}

\section*{Acknowledgments}

Alain Bensoussan is supported by the National Science Foundation under grant NSF-DMS-2204795. This work also constitutes part of Ziyu Huang's Ph.D. dissertation at Fudan University, China. Phillip Yam acknowledges the financial supports from HKGRF-14301321 with the project title ``General Theory for Infinite Dimensional Stochastic Control: Mean Field and Some Classical Problems'', and HKGRF-14300123 with the project title ``Well-posedness of Some Poisson-driven Mean Field Learning Models and their Applications''.

\appendix

\section{Proof of Lemma~\ref{lem:05}}\label{pf:lem:05}

We first prove \eqref{lem02_1}. Using the SDE in \eqref{intr_3} and Assumption (A1), we have for $s\in[t,T]$,
\begin{align*}
	\frac{d}{ds}\left\|X^V(s)\right\|^2_{\hr_m}&=2\left(X^V(s),F(X^V(s),V(s),s)\right)_{\hr_m}+\sum_{j=1}^n\left\|A^j (X^V(s),V(s),s)\right\|^2_{\hr_m}\\
	&\le C(L) \left(1+\left\|X^V(s)\right\|^2_{\hr_m}+\left\|V(s)\right\|_{\ur_m}^2\right).
\end{align*}
From Gr\"onwall's inequality, we obtain \eqref{lem02_1}. We then prove the estimate \eqref{lem04_1} for $\mathcal{D}_{\tv}X^V(s)$. By Equation \eqref{tx} and Assumption (A1), we have
\begin{align*}
	&\frac{d}{ds}\left\|\mathcal{D}_{\tv}X^V(s)\right\|^2_{\hr_m}\\
	=\ &2\left(\mathcal{D}_{\tv}X^V(s), D_X F\left(X^V(s),V(s),r\right)\left(\mathcal{D}_{\tv}X^V(s)\right)+D_V F\left(X^V(s),V(s),s\right)\left(\tv(s)\right)\right)_{\hr_m}\\
	&\quad +\sum_{j=1}^n \left\| D_X A^j\left(X^V(s),V(s),r\right)\left(\mathcal{D}_{\tv}X^V(s)\right)+D_V A^j\left(X^V(s),V(s),s\right)\left(\tv(s)\right) \right\|_{\hr_m}^2\\
	\le\ & C(L)\left(\left\|\mathcal{D}_{\tv}X^V(s)\right\|^2_{\hr_m}+ \left\|\tv(s)\right\|^2_{\ur_m}\right).
\end{align*}
By applying Gr\"onwall's inequality, we deduce \eqref{lem04_1}. We next prove \eqref{lem05_1}. It is obvious that $V^\epsilon\in L^2_{\mathcal{W}_{tX}}(t,T;\ur_m)$. We set $Y^\epsilon(s):=\frac{1}{\epsilon}\left(X^\epsilon(s)-X^V(s)\right)$. We first establish the uniform boundedness of the norm $\|Y^\epsilon(s)\|_{\hr_m}$. From the equation in \eqref{intr_3} and Assumption (A1), we have for $s\in[t,T]$,
\begin{equation}\label{delta_X}
	\begin{aligned}
		Y^\epsilon(s)=&\int_t^s \bigg[\int_0^1 D_X F\left(X^V(r)+\lambda\epsilon Y^\epsilon(r),V^\epsilon(r),r\right)(Y^\epsilon(r))d\lambda\\
		&\quad\qquad +\int_0^1 D_V F\left(X^V(r),V(r)+\lambda\epsilon\tv(r),r\right)\left(\tv(r)\right) d\lambda\bigg]dr\\
		&+\sum_{j=1}^n\int_t^s \bigg[\int_0^1 D_X A^j\left(X^V(r)+\lambda\epsilon Y^\epsilon(r),V^\epsilon(r),r\right)\left(Y^\epsilon(r)\right)d\lambda\\
		&\quad\qquad\qquad +\int_0^1 D_V A^j\left(X^V(r),V(r)+\lambda\epsilon\tv(r),r\right)\left(\tv(r)\right)d\lambda\bigg]dw_j(r).
	\end{aligned}
\end{equation}	
From Assumption (A1) and Cauchy's inequality, we have for $s\in(t,T]$,
\begin{align*}
	\|Y^\epsilon(s)\|^2_{\hr_m}&\le C(L,T)\left[\int_t^s\left( \|Y^\epsilon(r)\|^2_{\hr_m}+ \left\|\tv(r)\right\|^2_{\ur_m}\right)dr\right].
\end{align*}
By applying Gr\"onwall's inequality, we have 
\begin{align}\label{lem03_1}
	\sup_{t\le s\le T}\|Y^\epsilon(s)\|_{\hr_m}\le C(L,T)\left\|\tv\right\|_{L^2(t,T;\ur_m)}.
\end{align}	
We denote by $\Delta^\epsilon(s):=Y^\epsilon(s)-\tx^{V,\tv}(s)$. From \eqref{tx}, \eqref{delta_X} and Assumption (A1), we have for $s\in[t,T]$,
\begin{align*}
	\Delta^\epsilon(s)=&\int_t^s \bigg[\int_0^1 \Big[D_X F\left(X^V(r)+\lambda\epsilon Y^\epsilon(r),V^\epsilon(r),r\right)\left(Y^\epsilon(r)\right)\\
	&\quad\qquad\qquad -D_X F\left(X^V(r)+\lambda\epsilon Y^\epsilon(r),V^\epsilon(r),r\right)\left(\mathcal{D}_{\tv}X^V(r)\right)\Big]d\lambda\\
	&\quad\qquad +\int_0^1 \Big[D_X F\left(X^V(r)+\lambda\epsilon Y^\epsilon(r),V^\epsilon(r),r\right)\left(\mathcal{D}_{\tv}X^V(r)\right)\\
	&\quad\qquad\qquad\qquad -D_X F\left(X^V(r),V(r),r\right)\left(\mathcal{D}_{\tv}X^V(r)\right) \Big]d\lambda\\
	&\quad\qquad +\int_0^1 \Big[D_V F\left(X^V(r),V(r)+\lambda\epsilon\tv(r),r\right)\left(\tv(r)\right)\\
	&\quad\qquad\qquad\qquad -D_V F\left(X^V(s),V(s),s\right)\left(\tv(r)\right)\Big]d\lambda\bigg]dr\\
	&+\sum_{j=1}^n\int_t^s\bigg[\int_0^1 \Big[D_X A^j\left(X^V(r)+\lambda\epsilon Y^\epsilon(r),V^\epsilon(r),r\right)\left(Y^\epsilon(r)\right)\\
	&\qquad\qquad\qquad\qquad -D_X A^j\left(X^V(r)+\lambda\epsilon Y^\epsilon(r),V^\epsilon(r),r\right)\left(\mathcal{D}_{\tv}X^V(r)\right)\Big]d\lambda\\
	&\quad\qquad\qquad +\int_0^1 \Big[D_X A^j\left(X^V(r)+\lambda\epsilon Y^\epsilon(r),V^\epsilon(r),r\right)\left(\mathcal{D}_{\tv}X^V(r)\right)\\
	&\quad\qquad\qquad\qquad\qquad -D_X A^j\left(X^V(r),V(r),r\right)\left(\mathcal{D}_{\tv}X^V(r)\right)\Big]d\lambda\\
	&\quad\qquad\qquad +\int_0^1 \Big[D_V A^j\left(X^V(r),V(r)+\lambda\epsilon\tv(r),r\right)\left(\tv(r)\right)\\
	&\quad\qquad\qquad\qquad\qquad -D_V A^j\left(X^V(r),V(r),r\right)\left(\tv(r)\right)\Big]d\lambda\bigg]dw_j(r).
\end{align*}
From Assumption (A1) and Cauchy's inequality, we have    
\begin{align*}
	\|\Delta^\epsilon(s)\|^2_{\hr_m}\le C(L,T)&\left[\int_t^s \|\Delta^\epsilon(r)\|^2_{\hr_m} dr+\mathcal{I}(\epsilon)\right],
\end{align*}
where
\begin{align*}
	\mathcal{I}(\epsilon):=&\int_t^T\int_0^1 \bigg[\Big\|D_X F\left(X^V(s)+\lambda\epsilon Y^\epsilon(s),V^\epsilon(s),s\right)\left(\mathcal{D}_{\tv}X^V(s)\right)\\
	&\quad\qquad\qquad -D_X F\left(X^V(s),V(s),s\right)\left(\mathcal{D}_{\tv}X^V(s)\right)\Big\|^2_{\hr_m} \\
	&\qquad\qquad +\Big\|D_V F\left(X^V(s),V(s)+\lambda\epsilon\tv(s),s\right)\left(\tv(s)\right)\\
	&\qquad\qquad\qquad -D_V F\left(X^V(s),V(s),s\right)\left(\tv(s)\right)\Big\|^2_{\hr_m}\bigg] d\lambda ds\\
	&+\sum_{j=1}^n \int_t^T\int_0^1 \bigg[\Big\|D_X A^j\left(X^V(s)+\lambda\epsilon Y^\epsilon(s),V^\epsilon(s),s\right)\left(\mathcal{D}_{\tv}X^V(s)\right)\\
	&\qquad\qquad\qquad\qquad -D_X A^j\left(X^V(s),V(s),s\right)\left(\mathcal{D}_{\tv}X^V(s)\right)\Big\|^2_{\hr_m} \\
	&\quad\qquad\qquad\qquad +\Big\|D_V A^j\left(X^V(s),V(s)+\lambda\epsilon\tv(s),s\right)\left(\tv(s)\right)\\
	&\quad\qquad\qquad\qquad\qquad -D_V A^j\left(X^V(s),V(s),s\right)\left(\tv(s)\right)\Big\|^2_{\hr_m} \bigg] d\lambda ds.
\end{align*}
From Gr\"onwall's inequality, we have
\begin{align}\label{lem05_2}
	\sup_{t\le s\le T}\|\Delta^\epsilon(s)\|_{\hr_m}^2\le C(L,T)\mathcal{I}(\epsilon).
\end{align}
From Assumption (A1), estimates \eqref{lem04_1} and \eqref{lem03_1} and the dominated convergence theorem, we have
\begin{align}\label{lem05_3}
	\lim_{\epsilon\to0}I(\epsilon)=0.
\end{align}
From \eqref{lem05_2} and \eqref{lem05_3}, we have \eqref{lem05_1}.

\section{Proof of Lemma~\ref{lem:1}}\label{pf:lem:1}

From \eqref{G} and Assumption (B1), we have for $(X,V,s)\in \hr_m\times \ur_m\times[0,T]$,
\begin{align*}
	F(X,V,s)|_x&\le l \left(1+|X_x|+|V_x|+W_2(X\otimes m,\delta_0)\right)\le L(1+|X_x|+|V_x|+\|X\|_{\hr_m}),
\end{align*}
therefore, we have 
\begin{align}\label{lem1_1}
	\|F(X,V,s)\|_{\hr_m}\le C(l)(1+\|X\|_{\hr_m}+\|V\|_{\ur_m}).
\end{align}	
From \eqref{F} and Assumption (B2), we have
\begin{align*}
	|G(X,V,s)|\le C(l)\left(1+\|X\|_{\hr_m}^2+\|V\|_{\ur_m}^2\right).
\end{align*}
For $X,\tx\in \hr_m$ and $\epsilon\in(0,1)$, we have
\footnotesize
\begin{align*}
	&\frac{1}{\epsilon}\left[\left.F\left(X+\epsilon\tx,V,s\right)\right|_x-\left.F(X,V,s)\right|_x\right]\\
	=&\int_0^1\left(D_x f\left(X_x+\lambda\epsilon\tx_x,\left(X+\epsilon\tx\right)\otimes m,V_x,s\right)\right)^*\tx_x d\lambda\\
	&+\int_0^1\int_0^1\bar{\e}\bigg[\int_\brn \left(D_\xi \frac{d f}{d\nu}\left(X_x,(1-\lambda)X\otimes m+\lambda(X+\epsilon\tx)\otimes m,V_x,s\right)\left(\bar{X}_y+\delta\epsilon\bar{\tx}_y\right)\right)^*\bar{\tx}_y dm(y)\bigg]d\delta d\lambda.
\end{align*}
\normalsize
By the dominated convergence theorem, we obtain \eqref{lem1_2}. Similarly, by Fubini's theorem, we have
\begin{align*}
	&\lim_{\epsilon\to0}\frac{1}{\epsilon}\left[G\left(X+\epsilon\tx,V,s\right)-G(X,V,s)\right]\\
	=\ &\e\left[\int_{\brn} (D_x g(X_x,X\otimes m,V_x,s))^*\tx_x dm(x)\right]\\
	& +\e\left[\int_\brn \bar{\e} \left( \int_\brn \left(D_\xi \frac{d g}{d\nu}(X_x,X\otimes m,V_x,s)\left(\bar{X}_y\right)\right)^*\bar{\tx}_y dm(y)\right)dm(x)\right]\\
	=\ &\e\bigg[\int_\brn \bigg[D_x g(X_x,X\otimes m,V_x,s)\\
	&\qquad\qquad +\bar{\e}\left(\int_\brn D_\xi \frac{d g}{d\nu}(\bar{X}_y,X\otimes m,\bar{V}_y,s)(X_x) dm(y)\right)\bigg]^*\tx_x dm(x)\bigg],
\end{align*}
from which we deduce \eqref{lem2_2}. Arguments for \eqref{lem1_3}, \eqref{lem1_2'}, \eqref{lem1_3'}, \eqref{lem2_3} and \eqref{lem2_4} are similar. Other arguments for (A1) and (A2) are direct consequences of \eqref{lem1_2}-\eqref{lem2_4}. We next prove \eqref{lem3_2}. For any $(X,V,s,P,Q)\in \hr_m\times \ur_m \times[0,T]\times \hr_m\times\hr_m^n$ and $\tx\in \hr_m$, from Fubini's theorem, we have
\small
\begin{align*}
	&D_X\lr(X,V,s;P,Q)\left(\tx\right)\\
	=\ &\e\Bigg[\int_\brn \Bigg(P_x^* \left.D_X F(X,V,s)\left(\tx\right)\right|_x +\sum_{j=1}^n  (Q^j_x)^* \left.D_X A^j(X,V,s)\left(\tx\right)\right|_x \\
	&\qquad\qquad +\left(\left.D_X G(X,V,s)\right|_x\right)^* \tx_x \Bigg)dm(x)\Bigg]\\
	=\ &\e \left[\int_\brn P_x^* \left(D_x f(X_x,X\otimes m,V_x,s)\right)^*\tx_xdm(x)\right]\\
	& +\e\left[\int_\brn\bar{\e}\left( \int_\brn P_x^* \left(D_\xi \frac{d f}{d\nu}(X_x,X\otimes m,V_x,s)\left(\bar{X}_y\right)\right)^*\bar{\tx}_y dm(y)\right) dm(x)\right]\\
	& +\sum_{j=1}^n \e \left[\int_\brn (Q^j_x)^* \left(D_x \sigma^j(X_x,X\otimes m,V_x,s)\right)^* \tx_xdm(x)\right]\\
	& +\sum_{j=1}^n\e\left[\int_\brn\bar{\e}\left( \int_\brn (Q^j_x)^* \left(D_\xi \frac{d \sigma^j}{d\nu}(X_x,X\otimes m,V_x,s)\left(\bar{X}_y\right)\right)^*\bar{\tx}_y dm(y)\right) dm(x)\right]\\
	&+\e \bigg[ \int_\brn \bigg[D_x g(X_x,X\otimes m,V_x,s)\\
	&\quad\qquad\qquad +\bar{\e}\left(\int_\brn D_\xi \frac{d g}{d\nu}\left(\bar{X}_y,X\otimes m,\bar{V}_y,s\right)(X_x) dm(y)\right)\bigg]^* \tx_xdm(x)\bigg]\\
	=\ &\e \Bigg[ \int_\brn \Bigg(D_x f(X_x,X\otimes m,V_x,s)P_x +\sum_{j=1}^n  D_x \sigma^j(X_x,X\otimes m,V_x,s){Q^j_x}\\
	&\qquad\qquad +D_x g(X_x,X\otimes m,V_x,s)\\
	&\qquad\qquad +\bar{\e}\left( \int_\brn D_\xi \frac{d f}{d\nu}\left(\bar{X}_y,X\otimes m,\bar{V}_y,s\right)(X_x) P_x dm(y)\right)\\
	&\qquad\qquad +\sum_{j=1}^n\bar{\e}\left( \int_\brn  D_\xi \frac{d \sigma^j}{d\nu}\left(\bar{X}_y,X\otimes m,\bar{V}_y,s\right)(X_x) {Q^j_x} dm(y)\right)\\
	&\qquad\qquad+\bar{\e}\left(\int_\brn D_\xi \frac{d g}{d\nu}\left(\bar{X}_y,X\otimes m,\bar{V}_y,s\right)(X_x) dm(y)\right)\Bigg)^* \tx_xdm(x)\Bigg],
\end{align*}
\normalsize
from which we obtain \eqref{lem3_2}. Similarly, we have \eqref{lem3_3}.

\end{document}